\documentclass[11pt,reqno]{amsart}

\usepackage{graphicx}
\usepackage{mathrsfs}
\usepackage{color}
\usepackage{comment}
\usepackage{amssymb}
\usepackage{hyperref}

\newtheorem{prop}{Proposition}%[section]
\newtheorem{theo}[prop]{Theorem}
\newtheorem*{theo*}{Theorem}
\newtheorem{lemm}[prop]{Lemma}
\newtheorem{coro}[prop]{Corollary}

\newtheorem{defi}[prop]{Definition}
\newtheorem{rema}[prop]{Remark}

\theoremstyle{definition}

\newcommand{\NN}{\mathbf{N}}

\newcommand{\RR}{\mathbf{R}}

\newcommand{\ZZ}{\mathbf{Z}}

\newcommand{\cA}{\mathcal A}

\newcommand{\cC}{\mathcal C}

\newcommand{\cH}{\mathcal H}

\newcommand{\cQ}{\mathcal Q}

\DeclareMathOperator{\Lip}{Lip}
\DeclareMathOperator{\diam}{diam}

\DeclareMathOperator{\dist}{dist}

\DeclareMathOperator{\Div}{div}

\newcommand{\td}[2]{\frac{d #1}{d #2}}
\newcommand{\bangle}[1]{\left\langle #1 \right\rangle}

\newcommand{\ep}{\varepsilon}

\begin{document}

\title[Soap bubbles and positive scalar curvature]{Generalized soap bubbles and the topology of manifolds with positive scalar curvature}
\author{Otis Chodosh}
\address{OC: Department of Mathematics, Stanford University, Building 380, Stanford, CA 94305, USA}
\email{ochodosh@stanford.edu}
\author{Chao Li}
\address{CL: Department of Mathematics, Princeton University, Fine Hall, 304 Washington Road, Princeton, NJ 08540, USA}
\address{Current Address: Courant Institute, New York University, 251 Mercer St, New York, NY 10012, USA}
\email{chaoli@nyu.edu}

\begin{abstract}
We prove that for $n\in \{4,5\}$ a closed aspherical $n$-manifold does not admit a Riemannian metric with positive scalar curvature.

Additionally, we show that for $n\leq 7$, the connected sum of a $n$-torus with an arbitrary manifold does not admit a complete metric of positive scalar curvature. When combined with forthcoming contributions by Lesourd--Unger--Yau, this proves that the Schoen--Yau Liouville theorem holds for all locally conformally flat manifolds with non-negative scalar curvature.

A key geometric tool in these results are generalized soap bubbles---surfaces that are stationary for prescribed-mean-curvature functionals (also called $\mu$-bubbles). 
\end{abstract}

\maketitle

\section{Introduction}

We begin by recalling the following well-known result proven by R.\ Schoen and S.-T.\ Yau and by M.\ Gromov and B.\ Lawson.\footnote{D.\ Stern has recently discovered an interesting new proof of this when $n=3$ (see \cite{Stern}).} 
\begin{theo}[Geroch Conjecture \cite{SY:3d-torus,SY:descent,SY:sing-PMT,GromovLawson}]\label{thm.Geroch}
The $n$-torus does not admit a Riemannian metric of positive scalar curvature. 
\end{theo}
 This result (and generalizations thereof) has had several important consequences, including Schoen--Yau's proof of the positive mass theorem in general relativity \cite{SY:PMT1,Schoen:Yamabe-survey,SY:sing-PMT} and Schoen's resolution of the Yamabe problem concerning conformal deformations to constant scalar curvature \cite{Schoen:Yamabe}. 

In this paper, we provide two extensions of Theorem \ref{thm.Geroch}, ruling out positive scalar curvature on closed aspherical manifolds of dimensions $4$ and $5$, as well as complete metrics of positive scalar curvature on an arbitrary manifold connect sum with a torus. 

We prove these results by analyzing stable solutions to the prescribed mean curvature problem (called here $\mu$-bubbles). This seems to be the first use of these surfaces to yield global topological restrictions on positive scalar curvature; in Section \ref{subsec:idea-proofs} below, we discuss previous applications of $\mu$-bubbles for local geometry of scalar curvature (e.g. comparison theorems). We expect that $\mu$-bubbles will find other topological applications.

\subsection{Aspherical manifolds} Recall that a smooth manifold is aspherical if it has contractible universal cover.  For example, any closed manifold covered by a Cartan--Hadamard manifold (such as the torus) is aspherical. Our first main result is as follows.
\begin{theo}\label{thm.main}
For $n \in \{4,5\}$, a smooth closed aspherical n-manifold $N$ does not admit a smooth Riemannian metric with positive scalar curvature. Any metric of non-negative scalar curvature on $N$ is flat. 
\end{theo}

In a 1987 survey article \cite{SY:sketch}, Schoen--Yau first proposed that all closed aspherical manifolds do not admit metrics of positive scalar curvature.\footnote{S.-T. Yau has informed us that they in fact conjectured this in the early 1980s.} They outlined a proof of this in $4$-dimensions \cite[Theorem 6]{SY:sketch}, but many parts of the proof have not been given. Resolving these missing parts is essential to carry out their outline, as we do in this paper. Furthermore, the case of $5$-dimensional manifolds is considerably more involved. See Section 1.3 for a more detailed discussion.

Theorem \ref{thm.main} (in all dimensions $n\geq 4$) was also conjectured in a slightly different form by M.\ Gromov (see \cite[p.\ 113]{gromov1986large}). It has some link with the Novikov conjecture on topological invariance of certain polynomials of Pontryagin classes, as explained in \cite[p.\ 25]{gromov2019lectures}. Furthermore, as discussed in \cite[Section 16]{gromov2017questions}, Theorem \ref{thm.main} (and its conjectural higher dimensional analogue) is one of the central questions in the study of geometric and topological properties of manifolds with positive scalar curvature.

Recently, J. Wang proved Theorem \ref{thm.main} in the special case that $N$ is of dimension $4$ with nonzero first Betti number \cite[Chapter 7 (Theorem F)]{wang2019thesis}.\footnote{As pointed out in \cite{wang2019thesis}, there exist infinitely many aspherical $4$-manifolds with zero first Betti number \cite{RT}. } 

We note that in the first version of this article, we only considered Theorem \ref{thm.main} in the case of $n=4$. As we were finishing writing down the generalization to $n=5$, we received a preprint from M. Gromov \cite{Gromov2020metrics} containing a proof of the same generalization, as well as further extensions to certain non-compact $5$-dimensional manifolds. Both approaches seem relatively similar, but certain central technical steps (compare Section 6.3 and 6.4 in this paper to \cite[Section 6]{Gromov2020metrics}) are obtained via different methods (and were obtained independently). 

%A standard argument shows that Theorem \ref{thm.main} implies the following rigidity result:
%\begin{coro}\label{coro.rigidity}
%If a smooth closed aspherical Riemannian $4$-manifold $(N,g)$ has non-negative scalar curvature, then $(N,g)$ is flat. 
%\end{coro}

\subsection{The Geroch conjecture with arbitrary ends} Our second main result resolves a question arising in the work of Schoen--Yau on locally conformally flat manifolds \cite{SY:conf-flat} (cf.\ \cite[\S 6]{SY:lectures}).

\begin{theo}\label{thm.pmt-bend}
Let $n\leq 7$. For any $n$-manifold $X$, the connect sum $T^n\# X$ does not admit a complete metric of positive scalar curvature. The only complete metric of non-negative scalar curvature on $T^n\#X$ is flat. 
\end{theo}

For example, this implies that a punctured torus does not admit a complete metric of positive scalar curvature (note however, that this particular statement follows from work of Gromov-Lawson in \cite[Example 6.9]{GL:complete} based on relative index theorems). When $X$ is compact, Theorem \ref{thm.pmt-bend} is well-known: it has recently been proven in all dimensions by Schoen--Yau \cite{SY:sing-PMT} via an analysis of singular minimal surfaces (cf.\ \cite{SY:descent,GromovLawson}). 

We emphasize that M.\ Lesourd, R.\ Unger, and S.-T.\ Yau have recently proven \cite{LUY:liouville} Theorem \ref{thm.pmt-bend} for $n=3$. Their proof extends to $3<n\le 7$ under certain technical assumptions on the geometry and topology of the manifold at infinity. Our proof is different at a technical level (even when $n=3$) as compared to theirs (which involves analyzing non-compact stable minimal hypersurfaces).

The main reason to consider Theorem \ref{thm.pmt-bend} comes from the study of the topology of locally conformally flat manifolds with nonnegative scalar curvature. In their foundational work on these manifolds \cite{SY:conf-flat} (cf.\ \cite[Theorem 3.5]{SY:lectures}), Schoen--Yau have proven that the geometry of the developing map of a locally conformally flat manifold with non-negative scalar curvature is strongly influenced by the mass of the manifold obtained via a Green's function conformal blown-up (motivated by Schoen's solution to the Yamabe problem \cite{Schoen:Yamabe}). Such a blown-up manifold will have a distinguished asymptotically flat end, but will also have other ends with metrics that are complete but do not have any other constraints on their geometry or topology.

If there were no such uncontrolled ends, a well-known argument due to Lohkamp allows one to reduce the positive mass theorem to the Geroch conjecture for $T^n \# M$ (where $M$ is compact). Such a reduction is delicate for asymptotically flat manifolds with other complete ends, since the geometry along the other ends could affect certain global arguments. However, Lesourd--Unger--Yau have recently proven \cite{LUY:liouville} that, through a careful analysis of the Green's function and modification of Lohkamp's argument, one can reduce\footnote{For completeness, we outline this reduction in the appendix of this paper.} the study of the manifolds arising in the work of Schoen--Yau \cite{SY:conf-flat} to the situation in Theorem \ref{thm.pmt-bend}. As such, by combining Theorem \ref{thm.pmt-bend} with results in \cite{SY:conf-flat} and \cite{LUY:liouville}, we have the following definitive result. 
\begin{coro}\label{coro.SY-Liouville}
Suppose that $(M^n,g)$ is a complete Riemannian manifold with $R_g\geq 0$. If $\Phi : M \to S^n$ is a conformal map, then $\Phi$ is injective and $\partial\Phi(M)$ has zero Newtonian capacity. 
\end{coro}
As shown in \cite[Theorems 4.6 and 4.7]{SY:conf-flat}, the Liouville result in Corollary \ref{coro.SY-Liouville} has strong consequences for the higher homotopy groups of locally conformally flat manifolds with nonnegative scalar curvature.

Under additional assumptions on the geometry or dimension of $M$ (including, e.g., that $n\ge 7$), this was proven by Schoen--Yau \cite[Theorem 3.1]{SY:conf-flat}. In fact, Corollary \ref{coro.SY-Liouville} was announced by Schoen--Yau in full generality \cite[Proposition 4.4']{SY:conf-flat} (cf.\ \cite[p.\ 262]{SY:lectures}), but the essential ingredient along the lines of Theorem \ref{thm.pmt-bend} has never appeared. Finally, we reiterate that Lesourd--Unger--Yau \cite{LUY:liouville} have proven Corollary \ref{coro.SY-Liouville} when $n=3$ (as well as for $4\leq n\leq 7$ with certain additional assumptions on the geometry and topology of $(M,g)$ at infinity).

We remark that our proof of Theorem \ref{thm.pmt-bend} can be generalized to allow certain other manifolds in place of the torus, see Section \ref{section.SYS} and \cite{Chen:SYS}. We expect that one may remove the dimensional restriction (following \cite{SY:sing-PMT}) but we do not pursue this here.

\subsection{Idea of the proofs of Theorem \ref{thm.main} and Theorem \ref{thm.pmt-bend}}\label{subsec:idea-proofs}

We explain the idea of the proofs of the results described above. As will be seen later, both Theorem \ref{thm.main} and Theorem \ref{thm.pmt-bend} have the same central difficulty: it is hard to find a (compact) stable minimal surface due to non-compactness of the ambient manifold. Thus, we will mostly focus on Theorem \ref{thm.main} here, and briefly indicate the strategy of Theorem \ref{thm.pmt-bend} at the end.

Our strategy of proof of Theorem \ref{thm.main} extends the outline of Schoen--Yau in \cite{SY:sketch} when $n=4$. We now describe their outline (of Theorem \ref{thm.main} when $n=4$) and subsequently detail our contributions as well as how to generalize the strategy to $n=5$. 
\\

	Suppose to the contrary, that $(N^4,g)$ is a closed aspherical manifold with $R(g)\ge 1$. Pass to a non-compact cover $(\overline N,\overline g)$ which is homotopy equivalent to $S^1$. Fix $[\sigma]$ a generator of $\pi_1(\overline N)$. Minimize area among hypersurfaces dual to $[\sigma]$ to find a complete\footnote{In our paper, we find it convenient to instead work with a stable minimal hypersurface with fixed boundary that is very far from $\sigma$, but this is not an important point.} stable minimal hypersurface $M$.  By Schoen--Yau's inductive descent technique, the stable minimal hypersurface $(M,\overline g|_M)$ is Yamabe positive. 
	
	Schoen--Yau suggest that one can derive a contradiction from this as follows.  Suppose that one could find a large region $\Omega \subset M$ so that each component of $\partial\Omega$ (there might be many) has controlled area. One can hope that each component of $\partial\Omega$ can be filled in by some $3$-manifold in $\overline N$ of bounded diameter. The existence of such a fill-in is expected here, as $\overline N$ is a covering space of a closed manifold (cf.\ the Lemma below Theorem 6 in \cite{SY:sketch}). Given this, it would be possible to fill in each component of $\partial \Omega$ without affecting the intersection with $\sigma$. This would yield a $3$-cycle in $\overline N$ with non-trivial algebraic intersection with $\sigma$. This would contradict $H_3(\overline N) =0$, finishing the proof. 
%\end{proof}
\\

\begin{figure}[htbp]
	\centering
	\includegraphics[width=\textwidth]{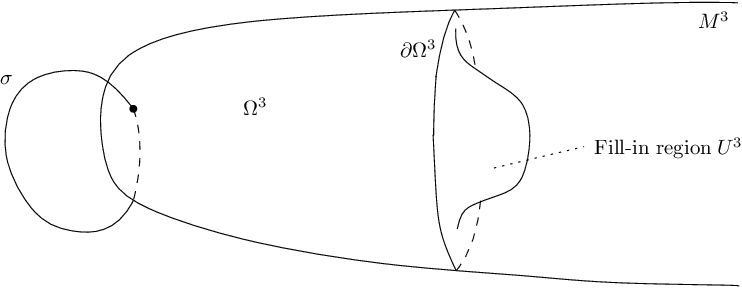}
	\caption{The Schoen--Yau outline}
	\label{pic1_SY_outline}
\end{figure}
In this paper, we obtain a complete proof of Theorem \ref{thm.main} (when $n=4$) along the lines of the outline given above. Our main contributions are the resolutions of the following fundamental difficulty arising in the outline above: to actually find $\Omega$ as stated. One intuitive way to approach this would be to exploit the fact that stable minimal surfaces in a three manifold with scalar curvature $\geq 2$ are spherical and have area $\leq 4\pi$ (see \cite{BBN} for a related rigidity result). However, there is no reason that $M$ will contain any minimal surfaces, much less admit an exhaustion by regions bounded by stable minimal surfaces. 

Instead, we rely on generalized soap bubbles---surfaces that are stationary for the prescribed-mean-curvature functional (called here the $\mu$-bubble after \cite{gromov1996positive,gromov2019lectures}). The use of this functional in scalar curvature problems was first considered by Gromov in \cite[Section 5$\frac56$]{gromov1996positive}. Recently, this approach was used by Gromov to obtain a list of comparison theorems in scalar curvature \cite{Gromov:metric-inequalities,gromov2019lectures}, and later expanded in work of J.\ Zhu \cite{zhu2020width,zhu2020rigidity}.\footnote{The prescribed-mean-curvature functional has recently been considered in other geometric problems as well, cf.\ \cite{ZhoiZhu:prescribed-CMC,BW:stable-PMC,Zhou:mult1}.} By applying this technique, we can localize the minimizer while still retaining the geometric properties used in the inductive descent argument.

%The second difficulty we resolve here is to find the desired fill-in of each component of $\partial \Omega$. %We prove the existence of such fill-ins as a Federer-Fleming type isoperimetric inequality in Proposition \ref{prop:FF-iso-cover}. Our proof works for regular covering spaces of a closed manifold. It would be interesting to see if such a property held without the regular cover assumption.

We also note that an additional issue with the outline as explained above is that the geometry of $(M,\overline g  |_M)$ after conformally deforming to positive scalar curvature might be dramatically different near infinity from the geometry of $(\overline N,\overline g)$. To handle this, we combine the $\mu$-bubble method with the warped product descent technique from \cite{SY:ric,SY:sing-PMT}. \\

Next, we briefly discuss the generalization of this argument to the case of $n=5$. We argue similarly to the above sketch and produce a curve $\sigma$ and $4$-dimensional minimal surface $M_4$  (we find it convenient to work in the universal cover and with a minimizing line $\sigma$, but this is not an important issue) with boundary far from $\sigma$, so that $M_4$ has non-trivial algebraic intersection with $\sigma$. We can then use the $\mu$-bubble argument as in $n=3$ to find a region in $M_4$ whose boundary is Yamabe positive and lies far away from $\sigma$. At this point, we encounter a serious difficulty not present in the previous situation: $3$-manifolds with positive scalar curvature can have arbitrarily large diameter. 

As such, we develop a slice-and-dice method that we use to decompose the $3$-dimensional $\mu$-bubble $M_3$ into pieces of controlled diameter and overlap. To do so, we first slice the $\mu$-bubble by $2$-dimensional spheres chosen to simplify $H_2(M_3)$ appropriately. Then, we dice the resulting manifold with boundary, using free boundary $\mu$-bubbles (if we did not slice first, the $\mu$-bubble argument would only produce a decomposition into regions with bounded distance to the boundary, but if there are many boundary components, such a region could have arbitrarily large diameter). Because we arrange that the free boundary $\mu$-bubbles are of disk type, this dices $M_3$ into connected components of controlled diameter, and so that the overlap of the sets (along their boundaries) has controlled diameter. This is sufficient to complete the fill-in argument explained for $n=4$ above.

We note that the slice-and-dice method can be used to show that a compact $3$-manifold with positive scalar curvature can be mapped to a ($1$-dimensional) graph by a map with pre-images of uniformly bounded diameter, or equivalently that such a manifold has bounded first Uryson width. See also \cite{GL:complete,Katz1988diameter,Guth2011volume} for related results. Y.\ Liokumovich and D.\ Maximo have recently obtained a related (but strictly more general) result along these lines; it would be interesting to compare the two approaches.  \\

Finally we briefly motivate the proof of Theorem \ref{thm.pmt-bend}. We again seek a $\mu$-bubble representing a suitable homology class and apply the Schoen-Yau descent argument (in fact, we argue slightly differently to this, by passing to a certain covering space so that we can consider boundaries rather than arbitrary homological relations). The key point in this argument is then to find the correct prescribing mean curvature function. 

A basic difficulty present here (as compared to the previous discussion) is that we are only assuming positivity of the scalar curvature rather than uniformly positive scalar curvature. The standard $\mu$-bubble technique (like the one we use in Theorem \ref{thm.main}) could potentially fail if---for example---the scalar curvature decayed faster than quadratically. However, because there is a sufficiently nice covering space, we can make the argument work (morally, this has to do with the fact that the torus has sufficiently complicated topology). We note that a similar argument (cf.\ Theorem \ref{thm.SYS} and \cite{Chen:SYS}) extends Theorem \ref{thm.pmt-bend} to manifolds in the form $M \# X$, where $M$ is a Schoen--Yau--Schick manifold (see \cite[Section 5]{Gromov:metric-inequalities}) and $X$ is arbitrary.  \\

The paper is organized as follows. We prove Theorem \ref{thm.main} in Section 2-6. In Section \ref{sec:top}, we provide some topological preliminaries concerning aspherical $n$-manifolds. In Section \ref{sec:warped-mu-bub} we discuss $\mu$-bubbles. and in Section \ref{sec:fb-mu-bub} we discuss the free boundary version of $\mu$-bubbles. In Section \ref{sec:diam-bds} we prove diameter bounds that will apply to stable $2$-dimensional $\mu$-bubbles (with  free boundary or not). Then in Section \ref{sec:proof-of-thm-main} we prove Theorem \ref{thm.main} by the slice-and-dice method mentioned above. We prove Theorem \ref{thm.pmt-bend} in Section \ref{sec:proof-pmt-bend}. In Section \ref{section.SYS}, we discuss an extension of Theorem \ref{thm.pmt-bend} to Schoen--Yau--Schick manifolds. \\

The authors were informed by X. Zhou that recently, S.-T. Yau and X. Zhou made some progress (independent of this paper) towards completing the Schoen--Yau outline of Theorem \ref{thm.main} in dimension $n=4$.

\subsection{Acknowledgements} We thank the referees for their helpful suggestions. We are grateful to Brian White useful discussions about isoperimetric inequalities, and to Boyu Zhang for several helpful conversations about $4$-manifold topology. We would like to thank Christos Mantoulidis, Davi Maximo, Misha Gromov, and Fernando Cod\'a Marques for their interest as well as Richard Schoen for bringing the $K(\pi,1)$ question to our attention and for his constant encouragement. 

We are grateful to Martin Lesourd for answering several questions concerning their work \cite{LUY:liouville}. Finally, the second named author wants to thank Xin Zhou for sharing his insights on the Schoen-Yau survey paper \cite{SY:sketch} in the early stages of this work.

O.C\  was supported in part by a Terman Fellowship, a Sloan Fellowship, NSF grants DMS-1811059/2016403 and 2304432. C.L.\ was supported by NSF grant DMS-2005287.

\section{Topological preliminaries}\label{sec:top}

In this section we collect some basic topological facts on closed aspherical manifolds. Let $N$ be a smooth connected $n$-manifold. $N$ is called aspherical, if its homotopy groups $\pi_j(N)$ is trivial for all integers $j > 1$. Equivalently, $N$ is a Eilenberg--MacLane space $K(\pi,1)$, where $\pi$ is the fundamental group of $N$. It is a standard result that $N$ is aspherical if and only if its universal cover, $\tilde{N}$, is contractible \cite[Theorem 4.5]{Hatcher2002AlgebraicTopology}. 

We fix $(N,g)$ a compact aspherical Riemannian $n$-manifold and denote by $\tilde N$ its universal cover. The following is standard. 
\begin{lemm}
$\tilde N$ is non-compact.
\end{lemm}
\begin{proof}
Because $\tilde N$ is contractible $H_n(\tilde N,\ZZ_2) = 0$. However, any compact $n$-manifold has $H_n(\tilde N,\ZZ_2) = \ZZ_2$.
\end{proof}

The next lemma is well known (and holds for the universal cover of any compact manifold). We recall the proof here for completeness. 

\begin{lemm}
There exists a geodesic line $\sigma : \RR \to (\tilde N,g)$ , i.e., $d(\sigma(t_1),\sigma(t_2)) = |t_2-t_1|$ for all $t_1,t_2\in\RR$. 
\end{lemm}
\begin{proof}
Fix $p \in \tilde N$, and choose $p_i \in \tilde N$ diverging. Let $\sigma_i$ denote minimizing geodesics between $p$ and $p_i$. Passing to a subsequence, $\sigma_i$ converge to $\sigma' : [0,\infty) \to\tilde N$ a minimizing ray. For $t_i\to\infty$ choose deck transformations $\Phi_i$ so that $d(p,\Phi_i(\sigma'(t_i)))$ is uniformly bounded. Then, $\sigma_i'(t) = \Phi_i(\sigma'(t+t_i))$ subsequentially converges to a geodesic line $\sigma$.
\end{proof}

Choose two smooth functions $\rho_1,\rho_2 : N\to\RR$ so that 
\[
|\rho_1(p) - d(p,\sigma([0,\infty))| \leq 1, \qquad |\rho_2(p) - d(p,\sigma(0))| \leq 1.
\]
For a large regular value of $\rho_1$, $L_1\gg1$, set $U : = \rho_1^{-1}((-\infty,L_1))$. By construction, $\partial U$ is a smooth properly embedded hypersurface. 

\begin{lemm}\label{lemm:intersection-sigma-M-0-bounded-set}
For $L_2$ sufficiently large, $\sigma \cap \partial U \subset \rho_2^{-1}((-\infty,L_2])$. 
\end{lemm}
\begin{proof}
Observe that $\sigma([0,\infty]) \subset U$.  Moreover, $\sigma((-\infty,-L_1-4]) \cap U = \emptyset$. Indeed, if $t_1 < -L_1-4$ has $\sigma(t_1) \in U$ then 
\[
d(\sigma(t_1),\sigma([0,\infty)) \leq L_1 + 2
\]
As such, there is $t_2 \geq 0$ with
\[
L_1+4 \leq |t_1 - t_2| = d(\sigma(t_1),\sigma(t_2)) \leq L_1 + 3.
\]
This is a contradiction. Thus $\sigma \cap \partial U$ is contained in a compact set. This completes the proof. 
\end{proof}

Fix a regular value $L_2 \gg L_1$ of the function $\rho_2|_{\partial U}$ and set 
\[
M = (\partial U) \cap \rho_2^{-1}((-\infty,L_2)).
\]
Note that $M$ is a smooth properly embedded compact oriented hypersurface with boundary. Perturb $\sigma$ slightly so that it intersects $M$ transversely (and $\sigma \cap \partial U = \sigma \cap M$) and we still have
\[
|\rho_1(p) - d(p,\sigma([0,\infty))| \leq 2, \qquad |\rho_2(p) - d(p,\sigma(0))| \leq 2.
\]
We now verify several properties of $M$ and $\sigma$, as illustrated by Figure \ref{pic2_construction_M}.

\begin{figure}[htbp]
	\centering
	\includegraphics[width=\textwidth]{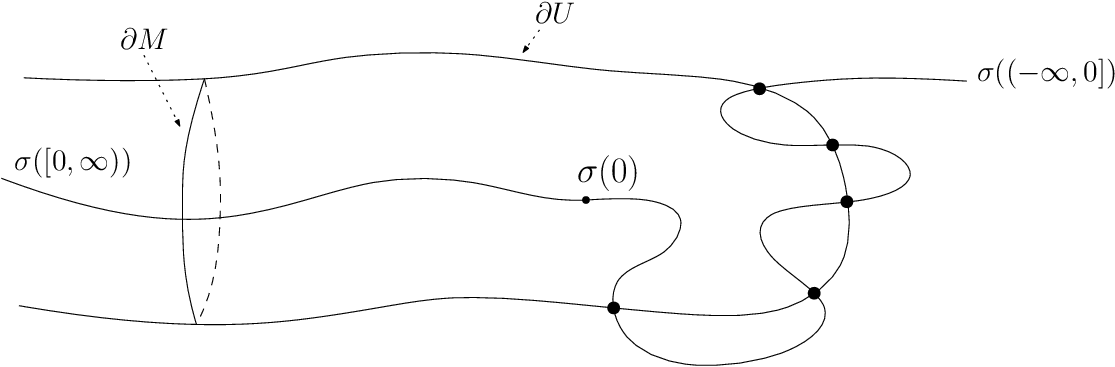}
	\caption{The construction of $\sigma$ and $M$}
	\label{pic2_construction_M}
\end{figure}

\begin{lemm}\label{lemm:intersection-sigma-M-0}
For $L_2$ sufficiently large, the curve $\sigma$ has non-zero algebraic intersection with $M$.
\end{lemm}
\begin{proof}
The curve $\sigma(t)$ leaves and then re-enters $U$ in (oppositely oriented) pairs until the smallest intersection time $t$, after which it never intersects $\partial U$. Because $\sigma \cap \partial U = \sigma \cap M$ this proves the assertion. 
\end{proof}

\begin{lemm}
For $L_2$ sufficiently large, $\partial M \neq \emptyset$ and $d(\partial M,\sigma(\RR)) \geq L_1 - 4$. 
\end{lemm}
\begin{proof}
Note that $\partial M \subset \rho_2^{-1}(L_2)$. First, suppose that $\partial M = \emptyset$. Then, by Lemma \ref{lemm:intersection-sigma-M-0}, we could conclude that $[M] \neq 0 \in H_{n-1}(\tilde N) = 0$. This is a contradiction. 

Suppose there is $t_2 \in \RR$ and $p\in \partial M$ with $d(p,\sigma(t_2)) \leq L_1 - 3$. Suppose that $t_2 \geq 0$. Then,
\[
\rho_1(p) \leq d(p,\sigma([0,\infty))) + 2 \leq L_1 - 1. 
\]
This cannot hold (since $\partial M \subset \rho_1^{-1}(L_1)$). As such, we see that $t_2 < 0$. Moreover, since $d(\partial M,\sigma(0)) \geq L_2 -2$, we find
\[
|t_2| = d(\sigma(t_2),\sigma(0)) \geq d(p,\sigma(0)) - d(p,\sigma(t_2)) \geq L_2 - L_1 + 1.
\]
On the other hand, we know there must be $t_1 \geq 0$ with $d(p,\sigma(t_1)) \leq L_1 + 2$. We have
\begin{multline*}
L_2 - L_1 + 1 \leq |t_1 - t_2| = d(\sigma(t_1),\sigma(t_2))  \\
\leq d(p,\sigma(t_1)) + d(p,\sigma(t_2)) \leq 2L_1 - 1. 
\end{multline*}
This is a contradiction as long as $L_2 > 3L_1 - 2$.
\end{proof}

\begin{prop}\label{prop:fill-bd-diam}
For $r>0$ there is $R=R(r)$ with the following property. Suppose that $\alpha$ is a $k$-cycle in $\tilde N$ with $\alpha \subset B_r(p)$ for some $p\in\tilde N$. Then $\alpha = \partial \beta$ for $\beta \subset B_R(p)$. 
\end{prop}
\begin{proof}
Fix $p_0\in\tilde N$. Since $\tilde N$ is contractible, for $r>0$, $B_r(p_0)$ is contractible in $B_{R_0(r)}(p_0)$ for some function $R_0=R_0(r)<\infty$. In particular $H_k(B_r(p_0)) \to H_k(B_{R_0}(p_0))$ is the zero map, for $k>0$. 

For any $p\in \tilde N$, there is a deck transformation $\Psi$ so that $d(p_0,\Psi(p)) \leq \diam N$. As such, $\Psi(B_r(p)) \subset B_{r+\diam N}(p_0)$. There is $\beta \in B_{R_0(r+\diam N)}(p_0) \subset B_{R_0(r+\diam N)+\diam N}(\Psi(p))$ with $\partial \beta = \Psi(\alpha)$. As such 
\[
\Psi^{-1}(\beta) \subset B_{R_0(r+\diam N)+\diam N}(p)
\]
has $\partial (\Psi^{-1}(\beta)) = \alpha$. This completes the proof. 
\end{proof}

\section{Warped $\mu$-bubbles}\label{sec:warped-mu-bub}
In this section we recall general existence and stability results for warped $\mu$-bubbles. For $n\leq 7$, consider $(M,g)$ a Riemannian $n$-manifold with boundary and assume that $\partial M = \partial_- M \cup \partial_+ M$ is a choice of labeling the components of $\partial M$ so that neither of the sets $\partial_\pm M$ are empty. Fix a smooth function $u > 0$ on $M$ and a smooth function $h$ on $\mathring M$ with $h\to \pm \infty$ on $\partial_\pm M$. 
Choose a Caccioppoli set $\Omega_0$ with smooth boundary $\partial\Omega_0 \subset \mathring M$ and $\partial_+ M\subset \Omega_0$. 
Consider the following functional
\begin{equation}\label{problem.variation}
\cA(\Omega)=\int_{\partial^* \Omega} u\, d\cH^{n-1} - \int_M (\chi_\Omega-\chi_{\Omega_0})hu \, d\cH^n,
\end{equation}
for all Caccioppoli sets $\Omega$ in $M$ with $\Omega\Delta \Omega_0\Subset \mathring M$. We will call $\Omega$ minimizing $\cA$ in this class a $\mu$-bubble. 
\begin{rema}
 Geometrically, this is equivalent to the functional
\begin{equation}
\tilde{\cA}(\Omega\times S^1)= \cH^{n-1}(\partial^*(\Omega\times S^1))- \int_{\Omega\times S^1} (\chi_{\Omega\times S^1}-\chi_{\Omega_0\times S^1})\tilde{h} \, d\cH^n,
\end{equation}
for $S^1$-invariant Caccioppoli sets $\Omega\times S^1$ inside $(M_1\times S^1,\tilde{g})$, where $\tilde{g}$ is the warped product metric $\tilde{g}=g+u^2 dt^2$, and $\tilde{h}$ is defined on $M\times S^1$ by $\tilde{h}(x,t)=h(x)$. However, we find it simplest to work with the form \eqref{problem.variation} instead of the warped product formulation. 
\end{rema}
\subsection{Existence of $\mu$-bubbles}  The existence and regularity of a minimizer of $\tilde{\cA}$ among all Caccioppoli sets (without any equivarient assumptions) was claimed by Gromov in \cite[Section 5.1]{gromov2019lectures}, and was rigorously carried out by Zhu in \cite[Proposition 2.1]{zhu2020width}. For the sake of completeness, we include a proof here.
\begin{prop}\label{prop.existence.regularity}
There exists a smooth minimizer $\Omega$ for $\cA$ such that $\Omega\Delta \Omega_0$ is compactly contained in the interior of $M$.
\end{prop}
\begin{proof}
	Let $\Omega$ be a Caccioppoli set in $M$ such that $\Omega\Delta \Omega_0 \Subset \mathring M$. By a standard approximation argument, we can assume that $\Omega$ has smooth boundary. We first show that, by adding to $\Omega$ a neighborhood of $\partial_+ M$, and subtracting from it a neighborhood of $\partial_-M$, one decreases $\cA_h$. For $\tau>0$, denote by $\Omega_\pm^{\tau}$ the distance-$\tau$ neighborhood of $\partial_\pm$ in $M$. Choosing $\tau$ sufficiently small, $\Omega_\pm^{\tau}$ has a foliation $\{S_\pm^{\rho}\}_{\rho\in [0,\tau]}$ by smooth equidistant hypersurfaces to $\partial_\pm M$. Denote by $\eta$ the unit normal vector field of $\{S_\pm^\rho\}$ defined in this $\Omega_\pm^\tau$, pointing into $M$ along $\partial_\pm M$. Let $\tau>0$ be sufficiently small so that 
	\[h u > \Div (u \eta) \text{ in }\Omega_+^\tau, \quad h u < - \Div (u\eta) \text{ in }\Omega_-^\tau.\]
	
	We compute
	\begin{align*}
	& \cA(\Omega\cup \Omega_+^\tau)-\cA(\Omega)\\
	&=\int_{\partial^*(\Omega\cup \Omega_+^\tau)}u \, d\cH^{n-1} - \int_{\partial \Omega} u \, d\cH^{n-1}- \int_{\Omega_+^\tau\setminus \Omega} hu \, d\cH^n\\
	&= \int_{\partial \Omega_+^\tau \setminus \Omega} u \, d\cH^{n-1}- \int_{\partial \Omega\cap \Omega_+^\tau} u \, d\cH^{n-1}- \int_{\Omega_+^\tau\setminus \Omega} hu \, d\cH^n\\
	& < \int_{\partial \Omega_+^\tau \setminus \Omega} u \, d\cH^{n-1}- \int_{\partial \Omega\cap \Omega_+^\tau} u \, d\cH^{n-1}- \int_{\Omega_+^\tau\setminus \Omega} \Div(u \eta) \, d\cH^n. 
\end{align*}
Moreover,
\begin{align*}
	\int_{\Omega_+^\tau\setminus \Omega} \Div(u\eta) \, d\cH^n&= \int_{\partial \Omega_+^\tau\setminus \Omega} (\eta\cdot \nu) u \, d\cH^{n-1} - \int_{\partial \Omega\cap \Omega_+^\tau} (\eta\cdot\nu) u \, d\cH^{n-1}\\
& \geq \int_{\partial\Omega^\tau_+\setminus\Omega} u \, d\cH^{n-1} - \int_{\partial\Omega \setminus\Omega^\tau_+} u \, d\cH^{n-1}.
	\end{align*}
	Thus $\cA(\Omega\cap \Omega_+^\tau)< \cA(\Omega)$. By an analogous calculation, $\cA(\Omega\setminus \Omega_-^\tau)<\cA(\Omega)$. Hence it suffices to consider the infimum of $\cA$ among $\cC$, where $\cC$ is the collection of Caccioppoli sets that contain $\Omega_+^\tau$ and are disjoin from $\Omega_-^\tau$. Since $|hu|<C_1$ in $M\setminus(\Omega_+^\tau\cup \Omega_-^\tau)$, we conclude that $\cA(\Omega)>-C_1 \cH^n(M)$ for all  $\Omega\in \cC$. Hence $I=\inf \{\cA(\Omega): \Omega\in \cC\}$ exists. Take a sequence $\Omega_k\in \cC$ with $\cA(\Omega_k)\rightarrow I$. Then $\cH^{n-1}(\partial \Omega_k)<C(I+C_1\cH^n(M))$. By BV-compactness, taking a subsequence, the sets $\Omega_k$ converge to a Caccioppoli set $\Omega$. It follows that $\Omega$ is a minimizer of $\cA$, and thus has smooth boundary by standard regularity theory \cite{Tam}.
\end{proof}

\subsection{Stability} We now discuss the first and second variation for a warped $\mu$-bubble. 
\begin{lemm}\label{lemm:1st-var}
If $\Omega_t$ is a smooth $1$-parameter family of regions with $\Omega_0 = \Omega$ and normal speed $\psi$ at $t=0$, then 
\[\td{}{t}\cA (\Omega_t)=\int_{\Sigma_t} (Hu + \bangle{\nabla_M u, \nu} - hu)\psi  \, d\cH^{n-1}\]
where $H$ is the scalar mean curvature of $\partial\Omega_t$ and $\nu$ is the outwards pointing unit normal. In particular, a $\mu$-bubble $\Omega$ satisfies \[
H = -u^{-1} \bangle{\nabla_M u,\nu} + h
\]
along $\partial\Omega$. 
\end{lemm}

\begin{lemm}\label{lemm:2nd-var}
Consider a $\mu$-bubble $\Omega$ with $\partial\Omega = \Sigma$. Assume that $\Omega_t$ is a smooth $1$-parameter family of regions with $\Omega_0 = \Omega$ and normal speed $\psi$ at $t=0$, then $\cQ(\psi):=\frac{d^2}{dt^2}\big|_{t=0}(\cA(\Omega_t))\ge 0$ where $\cQ(\psi)$ satisfies 
\begin{align*}
& \cQ(\psi) \\
&\leq \int_\Sigma \left(|\nabla_\Sigma \psi|^2 u -\tfrac12 (R_M - 1 -R_\Sigma +|\mathring{A_\Sigma}|^2)\psi^2 u+ (\Delta_M u  -\Delta_\Sigma u)\psi^2     \right.\\
&\left. \qquad\qquad -\tfrac 12 u^{-1} \bangle{\nabla_M u,\nu}^2\psi^2  - \tfrac 12 (1 + h^2 + 2 \bangle{\nabla_M h,\nu})\psi^2 u 
\vphantom{\mathring{A_\Sigma}}\right)d\cH^{n-1}.
\end{align*}
\end{lemm}
\begin{proof}
Differentiating the first variation, we find
\begin{align*}
    & \cQ(\psi) \\
    &= \int_\Sigma \left(-\psi u \Delta_\Sigma \psi -  \tfrac{1}{2} (R_M - R_\Sigma + |A_\Sigma|^2 + H_\Sigma^2)\psi^2 u  + H_\Sigma \langle \nabla_M u,\nu\rangle \psi^2  \right.\\
        &\left. \qquad\qquad +D^2u(\nu,\nu) \psi^2 - \langle\nabla_\Sigma u,\nabla_\Sigma \psi \rangle\psi- \langle \nabla_M(hu), \nu \rangle \psi^2 \vphantom{\tfrac12}\right) d\cH^{n-1}\\
      &=\int_\Sigma \left(|\nabla_\Sigma \psi|^2 u -\tfrac12 (R_M-R_\Sigma +|\mathring{A_\Sigma}|^2)\psi^2 u -\tfrac34 H_\Sigma^2 \psi^2 u \right.\\
        &\left. \qquad\qquad + (\Delta_M u  -\Delta_\Sigma u)\psi^2 - \bangle{\nabla_M h,\nu}u\psi^2 - h\bangle{\nabla_M u, \nu}\psi^2 \vphantom{\mathring{A_\Sigma}}\right)d\cH^{n-1}.
\end{align*}
Now use 
\[
\tfrac 12 H_\Sigma^2 \psi^2 u = \tfrac 12 u^{-1}\bangle{\nabla_M u,\nu}^2 \psi^2  -h \bangle{\nabla_M u,\nu} \psi^2  + \tfrac 12 h^2 \psi^2 u
\]
to write
\begin{align*}
& \cQ(\psi) \\
&\leq \int_\Sigma \left(|\nabla_\Sigma \psi|^2 u -\tfrac12 (R_M - 1 -R_\Sigma +|\mathring{A_\Sigma}|^2)\psi^2 u+ (\Delta_M u  -\Delta_\Sigma u)\psi^2     \right.\\
&\left. \qquad\qquad -\tfrac 12 u^{-1} \bangle{\nabla_M u,\nu}^2\psi^2  - \tfrac 12 (1 + h^2 + 2 \bangle{\nabla_M h,\nu})\psi^2 u 
\vphantom{\mathring{A_\Sigma}}\right)d\cH^{n-1}.
\end{align*}
This completes the proof. 
\end{proof}

\section{Free boundary warped $\mu$-bubbles}\label{sec:fb-mu-bub}

We will need a generalization of the previous discussion to the free boundary case. For $n\leq 7$, suppose that $(M^n,g)$ is a Riemannian $n$-manifold with co-dimension $2$ corners in the sense that any point has a neighborhood diffeomorphic to one of the following: $\RR^n$, $\{ x \in \RR^n : x^n\geq 0\}$ or $\{x \in \RR^n : x^{n-1},x^n \geq 0\}$. 

Assume that $M\setminus\mathring M = \partial_+M \cup \partial_- M \cup \partial_0 M$ where $\partial_\pm M,\partial_0 M$ are all smooth submanifolds of $M$ (possibly with boundary). We assume that $\partial_+ M \cap \partial_- M = \emptyset$ and $\partial_\pm M \cap\partial_0M$ consists of smooth co-dimension $2$ closed submanifolds. Moreover, we assume that $\partial_\pm M$ meets $\partial_0M$ orthogonally. 

For an open set $\Omega \subset M$ we denote by $\partial \Omega$ the \emph{topological boundary} in the sense that $\partial\Omega = \overline{\Omega} \cap \overline{\Omega^c}$. Observe that with this definition, if $\Omega =M$, $\partial\Omega = \emptyset$. 

Fix a function $u \in C^\infty(M)$ with $u>0$ and $h \in C^\infty(M\setminus (\partial_+M\cup\partial_-M))$. Assume that $h\to \pm \infty$ on $\partial_\pm M$. Consider a Caccioppoli set $\Omega_0$ with smooth boundary satisfying $\partial\Omega_0 \subset \mathring M \cup \partial_0 M$. 

Consider the functional (as before)
\[
\cA(\Omega) = \int_{\partial^* \Omega} u \, d\cH^{n-1} - \int_M (\chi_\Omega - \chi_{\Omega_0}) hu \, d\cH^n
\]
for all Caccioppoli sets $\Omega$ in $M$ with $\Omega\Delta \Omega_0 \Subset M\setminus (\partial_+M\cup\partial_-M)$. By similar arguments to the previous section, we can conclude the following result:
\begin{prop}\label{prop:fb-mu-bub}
There exists $\Omega$ with $\partial\Omega \subset \mathring M \cup \partial_0 M$ minimizing $\cA$ among such regions. The boundary $\partial\Omega$ is smooth and meets $\partial_0M$ orthogonally. We have
\[
H = - u^{-1}\bangle{\nabla_M u,\nu_{\partial\Omega}} + h
\]
along $\partial\Omega$. Finally, if $\Sigma$ is a component of $\partial\Omega$, then for any $\psi \in C^1(\Sigma)$, we have 
\begin{align*}
0 &\leq \int_\Sigma \left(|\nabla_\Sigma \psi|^2 u -\tfrac12 (R_M - 1 -R_\Sigma +|\mathring{A_\Sigma}|^2)\psi^2 u+ (\Delta_M u  -\Delta_\Sigma u)\psi^2     \right.\\
&\left. \qquad\qquad -\tfrac 12 u^{-1} \bangle{\nabla_M u,\nu_\Sigma}^2\psi^2  - \tfrac 12 (1 + h^2 + 2 \bangle{\nabla_M h,\nu_\Sigma})\psi^2 u 
\vphantom{\mathring{A_\Sigma}}\right)d\cH^{n-1}\\
& \qquad - \int_{\partial\Sigma} A_{\partial_0M} (\nu_\Sigma,\nu_\Sigma) \psi^2 u \, d\cH^{n-2} .
\end{align*}
\end{prop}
\begin{proof}
Existence and interior regularity follow in a similar manner to Proposition \ref{prop.existence.regularity} with the following modification. To show that a minimizer $\Omega$ exists such that $\Omega\Delta \Omega_0$ is disjoint from $\partial_\pm M$, for $\tau>0$ sufficiently small, we find a foliation $\{S^\rho_\pm\}_{\rho\in [0,\tau]}$ of a small neighborhood $\Omega_\pm^\tau$ of $\partial_\pm M$ such that each $S^\rho_\pm$ meets $\partial_0 M$ orthogonally. Denote by $\eta$ the unit normal vector field of $\{S_\pm^\rho\}$ pointing into $M$ along $\partial_\pm M$. Choose $\tau>0$ sufficiently small so that
\[hu>\Div (u\eta) \text{ in }\Omega_+^\tau, \quad hu < -\Div (u\eta) \text{ in }\Omega_{-}^\tau.\]
Then a similar computation as in the proof of Proposition \ref{prop.existence.regularity} (the boundary term from the divergence theorem vanishes since each $S_\pm^\rho$ meet $\partial_0 M$ orthogonally) shows that we always have $\cA(\Omega\cap \Omega_+^\tau)<\cA(\Omega)$ and $\cA(\Omega\setminus \Omega_{-}^\tau)<\cA(\Omega)$. In particular, this shows that the minimizer $\Omega$ is disjoint from the co-dimension $2$ corners. Thus, regularity at the free boundary follows by modifying the regularity results for free boundary isoperimetric surfaces in \cite{Gruter:partion} (see also \cite{Gruter:optimal}) so as to apply to the $\mu$-bubble functional. The second variation follows as before, but with the additional boundary term arising exactly as in the calculation of the second variation of free boundary minimal surfaces. 
\end{proof}

\section{Diameter bounds for certain surfaces} \label{sec:diam-bds}
In this section we derive diameter bounds for certain surfaces (in practice these will be stable $\mu$-bubbles in certain $3$-manifolds). Schoen--Yau \cite{SY:condensation} have proven that stable minimal surfaces in $3$-manifolds of positive scalar curvature satisfy such an inequality. Moreover, they have proven that the $1$-dimensional components of their minimal $k$-slicings satisfy a length bound \cite{SY:sing-PMT}. Since we will apply this slicing idea for $\mu$-bubbles (possibly with boundary), we find it convenient to state and prove the following results, analogous to the Schoen-Yau estimates \cite{SY:sing-PMT}. 

We use $\mu$-bubbles to give a slightly different proof of these diameter bounds as to compared to those in \cite{SY:condensation,SY:sing-PMT}. 

\begin{lemm}\label{lemm:closed-2-dim-diam-bds}
For closed $2$-dimensional Riemannian manifold $(\Sigma^2,g)$, suppose there is a smooth function $\lambda > 0$ so that 
\[
\Delta_\Sigma \lambda  \leq - (K_0-K_\Sigma) \lambda + \tfrac 12 \lambda^{-1} |\nabla_\Sigma \lambda|^2 
\]
for some $K_0 \in (0,\infty)$. Then $\diam_g \Sigma \leq \sqrt{\tfrac {2}{K_0}} \pi$. 
\end{lemm}
\begin{proof}

For contradiction, suppose there are $p,q \in \Sigma$ with 
\[
L : = d(p,q) >  L_0 : = \sqrt{\tfrac {2}{K_0}} \pi.
\]
There is $\ep>0$ and a smooth function 
\[
\rho : \Sigma \setminus (B_\ep(p)\cup B_\ep(q)) \to (-L_0/2,L_0/2)
\]
with $|\Lip \rho| \leq 1-\ep$ and $\rho \to  -L_0/2$ at $\partial B_\ep(p)$ and $\rho \to L_0/2$ at $\partial B_\ep(q)$. We then take 
\[
h(x) = - 2(1-\ep) \tfrac{\pi}{L_0}  \tan(\tfrac{\pi}{L_0} \rho(x)).
\]
Observe that 
\begin{align*}
& K_0 + \tfrac 12 h(x)^2 - |\nabla_\Sigma h|(x) \\
& \geq K_0 + 2(1-\ep)^2 \tfrac{\pi^2}{L_0^2} ( \tan^2(\tfrac{\pi}{L_0}\rho(x)) - \sec^2(\tfrac{\pi}{L_0}\rho(x)))\\
& = K_0 - 2(1-\ep)^2 \tfrac{\pi^2}{L_0^2} \\
& = K_0 - (1-\ep)^2 K_0 > 0. 
\end{align*} 
We will use this below.

Using Proposition \ref{prop.existence.regularity} we can find a $\mu$-bubble $\Omega$ minimizing
\[
\cA(\Omega) = \int_{\partial^*\Omega} \lambda \, d\cH^1 - \int_{\Omega}(\chi_\Omega-\chi_{\Omega_0}) h \lambda\, d\cH^2
\]
where $\Omega_0$ is some reference Caccioppoli set $B_\ep(p) \subset \Omega_0 \subset \Sigma \setminus B_\ep(q)$. By Lemma \ref{lemm:1st-var}, $k_\gamma = - \lambda^{-1}\bangle{\nabla_\Sigma \lambda,\nu_\gamma} + h$. As in Lemma \ref{lemm:2nd-var}, we compute
\begin{align*}
0  & \leq \int_\gamma \left( |\nabla_\gamma \psi|^2 \lambda - K_\Sigma  \psi^2 \lambda - k_\gamma^2 \psi^2 \lambda + (\Delta_\Sigma \lambda - \Delta_\gamma \lambda) \psi^2 \right.\\
        &\left. \qquad\qquad   - \bangle{\nabla_\Sigma \lambda,\nu_\gamma} \psi^2 h - \bangle{\nabla_\Sigma h ,\nu_\gamma} \psi^2 \lambda  \vphantom{\tfrac12}\right) d\cH^{1}.
\end{align*}
Take $\psi = \lambda^{-\frac 12}$ to find
\begin{align*}
0  & \leq \int_\gamma \left(\tfrac 1 4  \lambda^{-2} |\nabla_\gamma \lambda|^2  - K_\Sigma   - k_\gamma^2  +\lambda^{-1} (\Delta_\Sigma \lambda - \Delta_\gamma \lambda)  \right.\\
        &\left. \qquad\qquad  - \lambda^{-1} \bangle{\nabla_\Sigma \lambda,\nu_\gamma} h - \bangle{\nabla_\Sigma h ,\nu_\gamma}  \vphantom{\tfrac12}\right) d\cH^{1}.
\end{align*}
Using $\tfrac 12 k_\gamma^2 = \tfrac 12 \lambda^{-2} \bangle{\nabla_\Sigma \lambda,\nu_\gamma}^2 - \lambda^{-1} \bangle{\nabla_\Sigma\lambda,\nu_\gamma}h + \tfrac 12 h^2$, we have
\begin{align*}
0  & \leq \int_\gamma \left(\tfrac 1 4  \lambda^{-2} |\nabla_\gamma \lambda|^2  - K_\Sigma  +\lambda^{-1} (\Delta_\Sigma \lambda - \Delta_\gamma \lambda)  \right.\\
        &\left. \qquad\qquad -\tfrac 12 \lambda^{-2} \bangle{\nabla_\Sigma \lambda,\nu_\gamma}^2 - \tfrac 12 h^2 - \bangle{\nabla_\Sigma h ,\nu_\gamma}  \vphantom{\tfrac12}\right) d\cH^{1}\\
& = \int_\gamma \left( - \tfrac 3 4  \lambda^{-2} |\nabla_\gamma \lambda|^2  - K_\Sigma  +\lambda^{-1} (\Delta_\Sigma \lambda)  \right.\\
        &\left. \qquad\qquad -\tfrac 12 \lambda^{-2} \bangle{\nabla_\Sigma \lambda,\nu_\gamma}^2 - \tfrac 12 h^2 - \bangle{\nabla_\Sigma h ,\nu_\gamma}  \vphantom{\tfrac12}\right) d\cH^{1},
\end{align*}
where we integrated by parts in the second step. Now, using the equation satisfied by $\lambda$, we have
\begin{align*}
0  & \leq \int_\gamma \left( -\tfrac 3 4  \lambda^{-2} |\nabla_\gamma \lambda|^2  + \tfrac 12 \lambda^{-2} |\nabla_\Sigma\lambda|^2 -  \tfrac 12 \lambda^{-2} \bangle{\nabla_\Sigma \lambda,\nu_\gamma}^2    \right.\\
        &\left. \qquad\qquad- K_0 - \tfrac 12 h^2  - \bangle{\nabla_\Sigma h ,\nu_\gamma}  \vphantom{\tfrac12}\right) d\cH^{1}.
\end{align*}
Note that $|\nabla_\Sigma\lambda|^2 = |\nabla_\gamma\lambda|^2 + \bangle{\nabla_\Sigma \lambda,\nu_\gamma}^2$, so we find
\[
\int_\gamma (K_0 + \tfrac 12 h^2 + \bangle{\nabla_\Sigma h,\nu_\gamma}) d\cH^1 \leq 0
\]
However, we have seen above that $K_0 + \tfrac 12 h^2 + \bangle{\nabla_\Sigma h,\nu_\gamma}>0$. This is a contradiction, completing the proof. 
\end{proof}

The following is the free-boundary analogue of the previous result. We note that Carlotto--Franz \cite[Proposition 1.8]{CarlottoFranz} have recently extended the original Schoen--Yau method \cite{SY:condensation} to the setting of stable free boundary minimal surfaces. As above, we use $\mu$-bubbles to give an alternative proof of this fact.\footnote{In fact, we prove a slightly stronger statement than \cite{CarlottoFranz}, as needed for the warped product inductive descent argument from \cite{SY:sing-PMT}.}

\begin{lemm}\label{lemm:bdry-2-dim-diam-bds}
For compact $2$-dimensional Riemannian manifold $(\Sigma^2,g)$ with boundary, suppose there is a smooth function $\lambda > 0$ so that 
\[
\Delta_\Sigma \lambda  \leq - (K_0-K_\Sigma) \lambda + \tfrac 12 \lambda^{-1} |\nabla_\Sigma \lambda|^2 
\]
in $\Sigma$ for some $K_0 \in (0,\infty)$. Suppose also that $\bangle{\nabla_\Sigma \lambda,\eta} = - k_{\partial\Sigma} \lambda$ along $\partial\Sigma$ for $\eta$ the outwards pointing unit normal. Then, $\diam_g \Sigma \leq \sqrt{\tfrac {2}{K_0}} \pi$. 
\end{lemm}
\begin{proof}
If the diameter bound fails, we can argue precisely as in Lemma \ref{lemm:closed-2-dim-diam-bds} to find a free boundary $\mu$-bubble $\Omega$. Choose some curve $\gamma$ in $\partial \Omega$. If $\gamma$ is a closed loop, the calculations in Lemma \ref{lemm:closed-2-dim-diam-bds} yield a contradiction. As such, we assume that $\gamma : [a,b] \to \Sigma$ is an arc with $\partial\gamma \subset \partial\Sigma$. Using Proposition \ref{prop:fb-mu-bub}, the calculation above carries over to yield
\begin{align*}
0  & \leq \int_\gamma \left(\tfrac 1 4  \lambda^{-2} |\nabla_\gamma \lambda|^2  - K_\Sigma  +\lambda^{-1} (\Delta_\Sigma \lambda - \Delta_\gamma \lambda)  \right.\\
        &\left. \qquad\qquad -\tfrac 12 \lambda^{-2} \bangle{\nabla_\Sigma \lambda,\nu_\gamma}^2 - \tfrac 12 h^2 - \bangle{\nabla_\Sigma h ,\nu_\gamma}  \vphantom{\tfrac12}\right) d\cH^{1}\\
& \qquad - (k_{\partial\Sigma}(\gamma(b)) - k_{\partial\Sigma}(\gamma(a))) .
\end{align*}
At this point, we integrate by parts and pick up boundary terms 
\begin{align*}
0  & \leq \int_\gamma \left( - \tfrac 3 4  \lambda^{-2} |\nabla_\gamma \lambda|^2  - K_\Sigma  +\lambda^{-1} (\Delta_\Sigma \lambda)  \right.\\
        &\left. \qquad\qquad -\tfrac 12 \lambda^{-2} \bangle{\nabla_\Sigma \lambda,\nu_\gamma}^2 - \tfrac 12 h^2 - \bangle{\nabla_\Sigma h ,\nu_\gamma}  \vphantom{\tfrac12}\right) d\cH^{1}\\
& \qquad - ( \lambda(b)^{-1} \bangle{\nabla_\Sigma \lambda,\eta}(\gamma(b)) -  \lambda(a)^{-1} \bangle{\nabla_\Sigma \lambda,\eta}(\gamma(a)))\\
& \qquad - (k_{\partial\Sigma}(\gamma(b)) - k_{\partial\Sigma}(\gamma(a))) \\
& = \int_\gamma \left( - \tfrac 3 4  \lambda^{-2} |\nabla_\gamma \lambda|^2  + \tfrac 12 \lambda^{-2} |\nabla_\Sigma \lambda|^2   -\tfrac 12 \lambda^{-2} \bangle{\nabla_\Sigma \lambda,\nu_\gamma}^2  \right.\\
        &\left. \qquad\qquad  - K_0  - \tfrac 12 h^2 - \bangle{\nabla_\Sigma h ,\nu_\gamma}  \vphantom{\tfrac12}\right) d\cH^{1}.
        \end{align*}
The proof is now completed as before. 
\end{proof}

\section{Proof of Theorem \ref{thm.main}} \label{sec:proof-of-thm-main}

For $n=4,5$, consider $(N^n,g)$ a smooth closed aspherical $n$-manifold with a Riemannian metric of non-negative scalar curvature. Running the Ricci flow for a short time yields a metric with positive scalar curvature unless $g$ is Ricci flat. However, if $g$ is Ricci flat, the Cheeger--Gromoll splitting theorem \cite[Theorem 3]{CG:split} implies that the universal cover $(\tilde N,\tilde g)$ splits isometrically as $(\tilde N',\tilde g')\times \RR^k$ for $\tilde N'$ compact. Because $N$ is aspherical, $\tilde N'$ is a point, so $(\tilde N,\tilde g)$ is flat $\RR^n$. 

As such, it suffices to prove that $N$ does not admit a metric of positive scalar curvature. We will assume that $n=5$, since if $(N^4,g)$ is a compact aspherical $4$-manifold with positive scalar curvature, then $N \times S^1$ is aspherical (cf.\ \cite[Proposition 4.2]{Hatcher2002AlgebraicTopology}) and has positive scalar curvature when equipped with the product metric. 

To summarize the above discussion, we can assume that $(N^5,g)$ is a compact aspherical Riemannian $5$-manifold with scalar curvature $R_g \geq 5$. We write $\tilde N$ for the universal cover of $N$. By the results in Section \ref{sec:top}, there is a geodesic line $\sigma$ in the universal cover $\tilde N$ so that for any $L>0$ there is a compact two-sided hypersurface  with boundary $\hat M_4$ in $\tilde N$ with $d(\partial \hat M_4,\sigma(\RR)) \geq 3L$ and so that $\hat M_4$ has non-zero algebraic intersection with $\sigma$. We will take $L$ sufficiently large below. 

\subsection{The $\sigma$-transversal $4$-dimensional area-minimizer} Find a smooth two-sided compact area-minimizing hypersurface $M_4$ homologous to $\hat M_4$ relative to $\partial M_4 = \partial \hat M_4$. Stability of $M_4$ and $R_g \geq 5$ implies that 
\[
\int_{M_4} (|\nabla_{M_4} \psi|^2 - \tfrac 12 (5-R_{M_4}+ |A|^2)\psi^2 \, d\cH^4 \geq 0,\qquad \forall\psi\in C^1_0(M). 
\]
As such, there is $u_4 \in C^\infty_0(M_4)$, $u_4>0$ on $\mathring M_4$, with 
\begin{equation}\label{eq:u4-PDE}
\Delta_{M_4} u_4 \leq -\tfrac 12 (5-R_{M_4}) u_4 . 
\end{equation}
Our goal is to show that for $L$ sufficiently large, we can find $\Omega_4 \subset M_4$ so that $\partial\Omega_4$ can be filled in by $4$-chains that avoid $\sigma(\RR)$. Then, adding these chains to $\Omega_4$ we find $4$-cycle with non-trivial algebraic intersection with $\sigma$, contradicting $H_4(\tilde N) = 0$. 

\subsection{The $3$-dimensional $\mu$-bubble} Pick $\rho_4 : \tilde N \to \RR$ a smoothing of $d(\cdot,\sigma(\RR))$ so that $|\Lip \rho_4| \leq 2$, $\rho_4 \geq L + 4\pi + 1$ on $\partial M_4$ and $\rho_4 = 0$ on $M_4\cap\sigma(\RR)$.  Fix $\ep>0$ small so that $L+4\pi +\ep$ and $L-\ep$ are regular values of $\rho_4$ and then define 
\[
\tilde \rho_4(x) = \frac{\rho_4(x) - L - 2\pi}{4 + \frac 2\pi \ep}, \qquad h_4(x) = - \tan \left(\tilde\rho_4(x)\right)
\]
for $x \in M_4' : = \{\tilde \rho_4(x) \in (-\tfrac \pi 2,\tfrac \pi 2) \}$. Observe that 
\begin{multline}\label{eq:h4-mu-bubb-calc}
1 + h_4(x)^2 - 2 |\nabla h_4| \\
\geq 1 + \tan^2(\tilde \rho_4(x)) - \frac{2}{4+\frac 2 \pi \ep} |\Lip \rho_4(x)| \frac{1}{\cos^2(\tilde \rho_4(x))} \geq 0
\end{multline}
on $M_4'$. Choosing $a_4 \sim 0$ a regular value of $\tilde\rho_4$, we can set $\Omega_0 = \tilde\rho^{-1}((-\infty,a_4))$ and use Proposition \ref{prop.existence.regularity} to minimize
\[
\cA_4(\Omega) : = \int_{\partial\Omega} u_4 \, d\cH^3 - \int_{M_4} (\chi_{\Omega} - \chi_{\Omega_0}) h_4 u_4 \, d\cH^4
\]
among Caccioppoli sets $\Omega$ with $\Omega\Delta\Omega_0 \Subset M_4'$. Denote this minimizer by $\Omega_4$ and let $M_3: = \partial\Omega_4$. Note that $M_3$ is a cycle and 
\[
d(M_3,\sigma(\RR)) \geq \tfrac 12 (L-\ep).
\]

Using \eqref{eq:u4-PDE}, \eqref{eq:h4-mu-bubb-calc} and Lemma \ref{lemm:2nd-var} we see that
\[
\int_{M_3} \left( |\nabla_{M_3} \psi|^2 u_4 - \tfrac 12 (4-R_{M_3}) \psi^2 u_4 - (\Delta_{M_3} u_4) \psi^2 \right) \, d\cH^3 \geq 0
\]
for all $\psi \in C^1(M_3)$. In particular, there is $u_3$ with
\[
 \Div_{M_3} (u_4 \nabla_{M_3} u_3) \leq - \tfrac 12 (4-R_{M_3}) u_3 u_4 - (\Delta_{M_3} u_4)u_3.
\]
This implies that
\begin{align}\label{eq:deltaM3-u3u4}
\Delta_{M_3} (u_3u_4) & = \Div_{M_3}(u_3 \nabla_{M_3} u_4) +  \Div_{M_3}(u_4 \nabla_{M_3} u_3) \\
& \leq - \tfrac 12 (4-R_{M_3}) u_3 u_4 + \bangle{\nabla_{M_3} u_3,\nabla_{M_3} u_4}. \nonumber
\end{align}
As mentioned above, our goal is now to show that each component of $M_3$ can be filled by a chain with uniformly bounded diameter as $L\to\infty$. Note that $M_3$ is Yamabe positive but it might hold that the diameter of each component of $M_3$ is unbounded as $L\to\infty$. 

\subsection{A slice and dice procedure for $M_3$}\label{section.slice.dice} We can consider each component of $M_3$ separately, so it suffices to assume that $M_3$ is connected. It will be important later to observe that $M_3$ is orientable, by construction. 

Set $\lambda_3 : = u_3u_4>0$. For $\Sigma^2 \subset M_3$, we define 
\[
\hat \cA_3(\Sigma) = \int_\Sigma  \lambda_3 . 
\]
Arguing as in \cite{SY:sing-PMT}, we can derive the following result for stable critical points of $\hat\cA_3$. 
\begin{lemm}\label{lem:hat-cA-3-crit-pts}
Suppose that $\Sigma$ is a connected (two-sided) stable critical point of $\hat\cA_3$. Then $\Sigma$ is a topological sphere with $\cH^2(\Sigma) \leq 2\pi$ and $\diam \Sigma \leq \pi $. 
\end{lemm}
\begin{proof}
Lemma \ref{lemm:1st-var} implies that $H = - \lambda_3^{-1} \bangle{\nabla_{M_3} \lambda_3,\nu}$. Similarly, Lemma \ref{lemm:2nd-var} combined with \eqref{eq:deltaM3-u3u4} yields 
\begin{align*}
0 & \leq \int_\Sigma \left(|\nabla_\Sigma \psi|^2 \lambda_3 - (2   - K_\Sigma )\psi^2 \lambda_3 - (\Delta_\Sigma \lambda_3)\psi^2    \right.\\
&\left. \qquad\qquad + \bangle{\nabla_{M_3}u_3,\nabla_{M_3}u_4}\psi^2  -\tfrac 12 \lambda_3^{-1} \bangle{\nabla_{M_3} \lambda_3,\nu}^2\psi^2   
 \right)d\cH^{2}.
\end{align*}
Note that 
\[
 2 \lambda_3 \bangle{\nabla_{M_3}u_3,\nabla_{M_3}u_4} -  \bangle{\nabla_{M_3} \lambda_3,\nu}^2   \leq |\nabla_{M_3}\lambda_3|^2  - \bangle{\nabla_{M_3} \lambda_3,\nu}^2  = |\nabla_{\Sigma}\lambda_3|^2
\]
This yields
\begin{equation}\label{eq:u2-var-eqn}
0 \leq \int_\Sigma \left(|\nabla_\Sigma \psi|^2 \lambda_3 - (2   - K_\Sigma )\psi^2 \lambda_3 - (\Delta_\Sigma \lambda_3)\psi^2 + \tfrac 12 \lambda_3^{-1}  |\nabla_{\Sigma}\lambda_3|^2   \psi^2   
 \right)d\cH^{2}
\end{equation}
Take $\psi = \lambda_3^{-\frac 12}$ to find
\begin{align*}
0 & \leq \int_\Sigma \left(\tfrac 34 \lambda_3^{-2} |\nabla_\Sigma \lambda_3|^2  - (2   - K_\Sigma )  - \lambda_3^{-1}(\Delta_\Sigma \lambda_3)  \right)d\cH^{2}.
\end{align*}
Integrating by parts on the last term, this implies
\[
2 \cH^2(\Sigma) \leq \int_\Sigma K_\Sigma \, d\cH^2 = 2\pi \chi(\Sigma).
\]
As such, $\Sigma$ is a topological sphere and $\cH^2(\Sigma) \leq 2\pi$. 

Returning to \eqref{eq:u2-var-eqn}, we find a smooth function $u_2>0$ on $\Sigma$ with 
\[
\Div_\Sigma(\lambda_3 \nabla_\Sigma u_2) \leq - (2-K_\Sigma) u_2\lambda_3 - (\Delta_\Sigma\lambda_3)u_2  + \tfrac 12 \lambda_3^{-1}  |\nabla_{\Sigma}\lambda_3|^2    u_2.
\]
Set $\lambda_2 = u_2\lambda_3 = u_2u_3u_4$. We compute
\[
\Delta_\Sigma \lambda_2  \leq - (2-K_\Sigma) \lambda_2 + \bangle{\nabla_\Sigma u_2,\nabla_\Sigma \lambda_3}  + \tfrac 12 \lambda_3^{-1}  |\nabla_{\Sigma}\lambda_3|^2 .
\]
Now, we observe that 
\[
 \lambda_2^{-1} |\nabla_\Sigma \lambda_2|^2 \geq 2 \bangle{\nabla_\Sigma u_2,\nabla_\Sigma \lambda_3} + \lambda_3^{-1} |\nabla_\Sigma \lambda_3|^2,
\]
so 
\[
\Delta_\Sigma \lambda_2  \leq - (2-K_\Sigma) \lambda_2 + \tfrac 12 \lambda_2^{-1} |\nabla_\Sigma \lambda_2|^2 .
\]
As such, the diameter bounds follow from Lemma \ref{lemm:closed-2-dim-diam-bds}. 
\end{proof}

We first use $\hat\cA_3$-minimization to slice $M_3$ into a manifold with simple second homology. This will allow us to use free boundary $\mu$-bubbles to dice the resulting manifold into pieces that can be filled within bounded distance. 
\begin{lemm}
There are $\hat \Sigma_1,\dots,\hat \Sigma_k\subset M_3$ pairwise disjoint two-sided stable critical points of $\hat\cA_3$ so that the manifold with boundary $\hat M_3 := M_3 \setminus (\cup_{j=1}^k \hat\Sigma_j)$ is connected and has $H_2(\partial \hat M_3) \to H_2(\hat M_3)$ surjective. 
\end{lemm}
\begin{proof}
We proceed inductively. Assume that for $k\geq 0$ we have chosen 
\[
\hat\Sigma_1,\dots,\hat\Sigma_k \subset M_3
\]
pairwise disjoint two-sided stable critical points of $\hat\cA_3$. Suppose that for $\hat M_3^{(k)} := M_3 \setminus (\cup_{j=1}^k \hat\Sigma_j)$, the map $H_2(\partial \hat M_3^{(k)}) \to H_2(\hat M_3^{(k)})$ is not surjective. Consider $\alpha$ not in the image. Minimize $\hat\cA_3$ area in the homology class. Because each component of $\partial\hat M_3^{(k)}$ is smooth and stationary for $\hat \cA_3$, we can find a representative of $\alpha$ 
\[
\alpha = [\Sigma'_1] + \dots + [\Sigma'_\ell]
\]
with the connected surfaces $\Sigma_1',\dots,\Sigma_\ell'$ pairwise disjoint two-sided stable critical points. Moreover, each $\Sigma_j'$ is either contained in the interior of $\hat M_3^{(k)}$ or coincides with a component of $\partial\hat M_3^{(k)}$. By choice of $\alpha$, there must be some $j$ with $\Sigma_j'$ in the interior of $\hat M_3^{(k)}$, not separating, and not in the image of $H_2(\partial\hat M_3^{(k)})$. Define $\hat\Sigma_{k+1} = \Sigma'_j$. 

It suffices to show that this process terminates. If not, we obtain a sequence $\hat \Sigma_1,\dots,\hat \Sigma_k,\dots$ of stable critical points of $\hat\cA_3$. By Lemma \ref{lem:hat-cA-3-crit-pts}, each $\hat\Sigma_k$ has bounded area. Moreover, by a standard blow-up argument using \cite{fischer-colbrie-schoen,docarmo-peng,pogorelov}, the $\hat\Sigma_k$ have uniform curvature bounds. As such, passing to a subsequence, the $\hat\Sigma_{k_1},\hat\Sigma_{k_2},\dots$ are converging as one-sheeted graphs to some $\hat\Sigma_\infty$. However, this clearly contradicts the construction of the $\hat\Sigma_j$ above, since it shows that $\hat\Sigma_{k_{j+1}}$ is homologous to $\hat\Sigma_{k_{j}}$ in $\hat M_3^{(k_{j+1})}$, for $j$ large. This completes the proof.
\end{proof}

Choose $\hat M_3$ as in the previous lemma. We recall that for $\Omega \subset \hat M_3$, we write $\partial\Omega$ for the topological boundary. Suppose that $\Omega \subset \hat M_3$ is a connected region with corners and $\Omega\neq M$. We assume that $\partial \Omega$ consists of smooth properly embedded surfaces in $\hat M_3$.
\begin{lemm}\label{lemm:slice-and-dice-top}
A connected component of $\hat M_3 \setminus \Omega$ contains exactly one component of $\partial \Omega$. 
\end{lemm}
\begin{proof}
Clearly a component of $\hat M_3 \setminus \Omega$ contains at least one component of $\partial  \Omega$. Suppose, for contradiction, that some component $\Omega'$ of $\hat M_3 \setminus \Omega$ contained two components of $\partial'\Omega$, $\Sigma_1$ and $\Sigma_2$. For $j=1,2$, take a point $p_j\in \Sigma_j$. Then we can find curves $\gamma\in \Omega$, $\gamma'\in \Omega'$ connecting $p_1, p_2$. The concatenated curve $\sigma=\gamma\overline{\gamma}'$ is then an embedded $S^1$ that intersects $\Sigma_1$ transversely in precisely one point. This implies that $[\sigma]$ is not torsion in $H_1(\hat M_3)$. 

On the other hand, the long exact sequence in homology for a pair yields
\[
H_2(\partial \hat M_3) \to  H_2(\hat M_3)\to H_2(\hat M_3,\partial \hat M_3) \to H_1(\partial \hat M_3) = 0.
\]
The final term vanishes since $\partial\hat M_3$ consists of spheres, by Lemma \ref{lem:hat-cA-3-crit-pts}. Thus we conclude that $H_2(\hat M_3,\partial M_3) = 0$. Lefschetz duality (cf.\ \cite[Theorem 3.43]{Hatcher2002AlgebraicTopology}) implies that $H^1(\hat M_3) = 0$, so the universal coefficients theorem implies that $H_1(\hat M_3)$ is torsion. This is a contradiction, finishing the proof. 
\end{proof}

We now show how to inductively dice $\hat M_3$ by finitely many (guaranteed by the condition (1) below, since $\hat M_3$ is compact) free boundary $\mu$-bubbles of controlled diameter and boundary behavior. Fix $p$ in the interior of $\hat M_3$. For technical reasons, we start by choosing $\Omega_1 = B_\ep(p)$, chosen so that $B_\ep(p)$ is contained in the interior of $\hat M_3$. Assume that we have chosen regions 
\[
\Omega_1 \subset \Omega_2 \subset \dots \subset \Omega_k
\]
with the following properties:
\begin{enumerate}
\item $d_{\hat M_3}(\partial\Omega_{j+1},\partial\Omega_j) \ge \tfrac{2\pi}{3}$,
\item each component of $\Omega_{j+1}\setminus \Omega_j$ has diameter at most $10\pi$
\item any component $\Upsilon\subset \partial\Omega_j$ has $\diam \Upsilon \leq \pi$, and
\item each component of $\partial\Omega_j$ is either a topological sphere or a topological disk with boundary in $\partial\hat M_3$.
\end{enumerate}

Assume that $\Omega_k \neq \hat M_3$ and that there is some $p \in \hat M_3$ with $d(p,\Omega_k) > 4\pi$ (otherwise we set $\Omega_{k+1} = \hat M_3$). 

We now choose a $\mu$-bubble based on a smoothing $\rho$ of $d(\cdot,\Omega_k)$ such that $d(\cdot, \Omega_k)\le \rho\le \tfrac 32 d(\cdot, \Omega_k)$, $|\Lip \rho| < \tfrac 32$ and $\rho|_{\Omega_k} = 0$ and then taking
\[
h(x) = -\tan ( \tfrac 1 3 (\rho(x)-\pi) - \tfrac \pi 2 ).
\]
Observe that 
\[
1 + h(x)^2 - 2|\nabla h| \geq 0,
\]
by a similar calculation to the one used above. 

By Proposition \ref{prop:fb-mu-bub}, we can find a free boundary $\mu$-bubble $\Omega_{k+1}$ inside $\{\pi\le \rho\le 4\pi\}$ minimizing 
\[
\cA_3(\Omega) = \int_{\partial^*\Omega} \lambda_3 \, d\cH^2 - \int_\Omega (\chi_\Omega - \chi_{\Omega_0}) \lambda_3 h \, d\cH^3.
\] 
Choose a component $\Sigma$ of $\partial \Omega_{k+1}$. Clearly condition (1) is satisfied since $h\to \infty$ at the surface $\{\rho=\pi\}$, which satisfies $d(\cdot,\Omega_k)\ge \tfrac{2\pi}{3}$.

\begin{figure}[htbp]
	\centering
	\includegraphics[width=\textwidth]{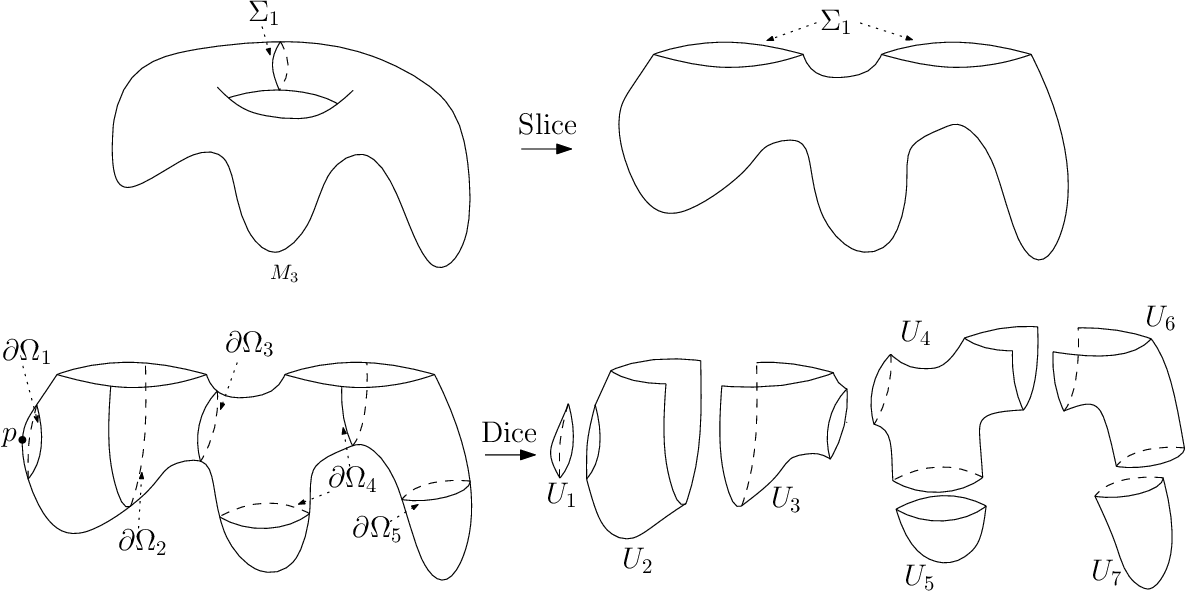}
	\caption{A slice and dice procedure on an $S^2\times S^1$}
	\label{pic3_slice_dice}
\end{figure}
\begin{lemm}
$\Omega_{k+1}$ satisfies property (2). 
\end{lemm}
\begin{proof}
By Lemma \ref{lemm:slice-and-dice-top}, if $p,q$ are in the same connected component of $\Omega_{k+1}\setminus \Omega_k$, then they are connected there to a unique component of $\partial\Omega_k$. Call this component $\Sigma_k$. By construction $d(p,\Sigma_k),d(q,\Sigma_k) \leq 4\pi$. Moreover, $\diam\Sigma_k \leq \pi$. Thus, $d(p,q) \leq 9\pi$. This completes the proof. 
\end{proof}

We now verify that $\Omega_{k+1}$ satisfies conditions (3) and (4). Take a component $\Sigma$ of $\partial \Omega_{k+1}$. By Proposition \ref{prop:fb-mu-bub}, $\Sigma$ has (possibly empty) free boundary at $\partial\hat M_3$. If $\Sigma$ has no boundary, then $\Sigma$ satisfies \eqref{eq:deltaM3-u3u4}, and thus by the same proof as in Lemma \ref{lem:hat-cA-3-crit-pts}, $\Sigma$ is a topological sphere with $\diam \Sigma\leq \pi$. If $\Sigma$ does have boundary, we can conclude that 
\begin{align*}
0 &\leq \int_\Sigma \left(|\nabla_\Sigma \psi|^2 \lambda_3 -\tfrac12 (R_{\hat M_3} - 1 - 2K_\Sigma  )\psi^2 \lambda_3 + (\Delta_{\hat M_3} \lambda_3  -\Delta_\Sigma \lambda_3)\psi^2     \right.\\
&\left. \qquad\qquad -\tfrac 12 \lambda_3^{-1} \bangle{\nabla_{\hat M_3} \lambda_3,\nu_\Sigma}^2\psi^2  
\vphantom{\mathring{A_\Sigma}}\right)d\cH^{2}\\
& \qquad - \int_{\partial\Sigma} A_{\partial \hat M_3} (\nu_\Sigma,\nu_\Sigma) \psi^2 \lambda_3 \, d\cH^{1} .
\end{align*}
Using \eqref{eq:deltaM3-u3u4} (recall that $\lambda_3 = u_3u_4$), we find
\begin{align*}
0 &\leq \int_\Sigma \left(|\nabla_\Sigma \psi|^2 \lambda_3 -  (1 - K_\Sigma  )\psi^2 \lambda_3 - (\Delta_\Sigma \lambda_3)\psi^2     \right.\\
&\left. \qquad\qquad + \bangle{\nabla_{\hat M_3}u_3,\nabla_{\hat M_3}u_4}\psi^2 -\tfrac 12 \lambda_3^{-1} \bangle{\nabla_{\hat M_3} \lambda_3,\nu_\Sigma}^2\psi^2  
\vphantom{\mathring{A_\Sigma}}\right)d\cH^{2}\\
& \qquad - \int_{\partial\Sigma} A_{\partial \hat M_3} (\nu_\Sigma,\nu_\Sigma) \psi^2 \lambda_3 \, d\cH^{1} .
\end{align*}
Taking first $\psi = \lambda_3^{-\frac 12}$, we conclude that 
\begin{align*}
0 &\leq \int_\Sigma \left(\tfrac 1 4 \lambda_3^{-2} |\nabla_\Sigma \lambda_3|^2  -  (1 - K_\Sigma  )  - \lambda_3^{-1} (\Delta_\Sigma \lambda_3)     \right.\\
&\left. \qquad\qquad + \lambda_3^{-1}\bangle{\nabla_{\hat M_3}u_3,\nabla_{\hat M_3}u_4} -\tfrac 12 \lambda_3^{-2} \bangle{\nabla_{\hat M_3} \lambda_3,\nu_\Sigma}^2 
\right)d\cH^{2}\\
& \qquad - \int_{\partial\Sigma} A_{\partial \hat M_3} (\nu_\Sigma,\nu_\Sigma)   \, d\cH^{1} \\
& = \int_\Sigma \left( K_\Sigma - 1  -\tfrac 3 4 \lambda_3^{-2} |\nabla_\Sigma \lambda_3|^2 + \lambda_3^{-1}\bangle{\nabla_{\hat M_3}u_3,\nabla_{\hat M_3}u_4} \right.\\
& \qquad \qquad \left. -\tfrac 12 \lambda_3^{-2} \bangle{\nabla_{\hat M_3} \lambda_3,\nu_\Sigma}^2 
\right)d\cH^{2}\\
& \qquad - \int_{\partial\Sigma} (\lambda_3^{-1} \bangle{\nabla_\Sigma\lambda_3,\eta} + A_{\partial \hat M_3} (\nu_\Sigma,\nu_\Sigma)  ) \, d\cH^{1}\\
& = \int_\Sigma \left( K_\Sigma - 1  -\tfrac 3 4 \lambda_3^{-2} |\nabla_\Sigma \lambda_3|^2 + \lambda_3^{-1}\bangle{\nabla_{\hat M_3}u_3,\nabla_{\hat M_3}u_4} \right.\\
&\left. \qquad \qquad -\tfrac 12 \lambda_3^{-2} \bangle{\nabla_{\hat M_3} \lambda_3,\nu_\Sigma}^2 
\right)d\cH^{2}\\
& \qquad - \int_{\partial\Sigma}( - \lambda_3^{-1} \bangle{\nabla_\Sigma\lambda_3,\nu_{\partial \hat M_3}} + H_{\partial\hat M_3} - A_{\partial \hat M_3} (\tau_{\partial \Sigma},\tau_{\partial \Sigma}) ) \, d\cH^{1}\\
& = \int_\Sigma \left( K_\Sigma - 1  -\tfrac 3 4 \lambda_3^{-2} |\nabla_\Sigma \lambda_3|^2 + \lambda_3^{-1}\bangle{\nabla_{\hat M_3}u_3,\nabla_{\hat M_3}u_4} \right.\\
&\qquad \qquad \left.-\tfrac 12 \lambda_3^{-2} \bangle{\nabla_{\hat M_3} \lambda_3,\nu_\Sigma}^2 
\right)d\cH^{2}\\
& \qquad + \int_{\partial\Sigma} A_{\partial \hat M_3} (\tau_{\partial \Sigma},\tau_{\partial \Sigma})  \, d\cH^{1}.
\end{align*}
Here $\eta$ is the outward conormal vector of $\partial \Sigma\subset \Sigma$, $\tau_{\partial \Sigma}$ is the unit tangential vector along $\partial \Sigma$. Since $\Sigma$ meets $\partial \hat M_3$ orthogonally, $A_{\partial \hat M_3}(\tau_{\partial \Sigma},\tau_{\partial \Sigma})=k_{\partial \Sigma}$. Also, as in Lemma \ref{lem:hat-cA-3-crit-pts}, the terms involving $\lambda_3$ are non-positive. Therefore we find 
\[
\cH^2(\Sigma) \leq \int_\Sigma K_\Sigma  + \int_{\partial\Sigma} k_{\partial\Sigma}  \, d\cH^{1} = 2\pi\chi(\Sigma). 
\]
This implies that $\Sigma$ is a topological disk.

Finally, as in Lemma \ref{lem:hat-cA-3-crit-pts} we find $u_2$ satisfying 
\[
\begin{cases}
\Div_\Sigma(\lambda_3 \nabla_\Sigma u_2) \leq - (2-K_\Sigma) u_2\lambda_3 - (\Delta_\Sigma\lambda_3)u_2  + \tfrac 12 \lambda_3^{-1}  |\nabla_{\Sigma}\lambda_3|^2 u_2 & \textrm{in $\Sigma$} \\
\bangle{\nabla_\Sigma u_2,\eta} = A_{\partial \hat M_3}(\nu_\Sigma,\nu_\Sigma)u_2 & \textrm{on $\partial\Sigma$}
\end{cases}
\]
Thus, we have that $\lambda_2=u_2\lambda_3 =u_2u_3u_4$ satisfies
\[
\Delta_\Sigma \lambda_2  \leq - (2-K_\Sigma) \lambda_2 + \tfrac 12 \lambda_2^{-1} |\nabla_\Sigma \lambda_2|^2 .
\]
in $\Sigma$ and 
\begin{align*}
\bangle{\nabla_\Sigma \lambda_2,\eta}  & =  A_{\partial \hat M_3}(\nu_\Sigma,\nu_\Sigma)  \lambda_2 + \bangle{\nabla_\Sigma \lambda_3,\eta} u_2 \\
& =   H_{\partial\hat M_3} \lambda_2  - \lambda_3^{-1} \bangle{\nabla_\Sigma \lambda_3,\nu_{\partial\hat M_3}} \lambda_2   -  k_{\partial\Sigma} \lambda_2 \\
& =  -  k_{\partial\Sigma} \lambda_2.
\end{align*}
Thus, Lemma \ref{lemm:bdry-2-dim-diam-bds} implies that $\diam \Sigma \leq \pi$.

\subsection{Filling $M_3$} We now show that $M_3$ can be filled in $\tilde N$ by a $4$-chain of uniformly bounded diameter (as $L\to\infty$). As explained above, this will give the desired contradiction, proving Theorem \ref{thm.main}. We first summarize the slice and dice procedure used above as it applies to $M_3$.

There exists a set of disjoint embedded spheres $\hat \Sigma_1,\dots,\hat\Sigma_k \subset M_3$ with $\diam\hat\Sigma_j \leq \pi$. Moreover, there is a set $\hat \Upsilon_1,\dots,\hat \Upsilon_\ell \subset M$ of embedded disks with $\diam \hat \Upsilon_j\leq \pi$ and $\partial \hat \Upsilon_j$ contained in $\cup_{i=1}^k \hat\Sigma_{i}$ and so that the interiors of $\hat \Upsilon_j$ are pairwise disjoint with each other as well as with each $\hat\Sigma_j$. Finally, each connected component $U_1,U_2,\dots,U_m$  of 
\[
M_3 \setminus ((\cup_{j=1}^k \hat\Sigma_j) \cup (\cup_{j=1}^\ell \hat\Upsilon_j))
\]
has diameter bounded by $10\pi$. 

Observe that any component of $\partial U_j$ is a topological sphere (it will be smooth except it could have a right angle corner along one or two closed curves arising from points where the free boundary disks intersect the original spheres) of extrinsic diameter at most $3\pi$. Write these spheres as $\Gamma_j^1,\dots,\Gamma_j^{n(j)}$. By Proposition \ref{prop:fill-bd-diam}, there is $R>0$ (independent of $L$) so that we can fill the $\Gamma_j^i$ by $\hat \Gamma_j^i$ with extrinsic diameter at most $R$. Then, 
\[
U_j - \hat\Gamma_j^1 - \dots - \hat \Gamma_j^{n(j)} 
\]
is a cycle with extrinsic diameter at most $2R+10\pi$. As such, there is $\hat R>0$ (independent of $L$) so that by Proposition \ref{prop:fill-bd-diam} we can find a $4$-chain $\hat U_j$ of extrinsic diameter at most $\hat R$ with
\[
\partial \hat U_j = U_j - \hat\Gamma_j^1 - \dots - \hat \Gamma_j^{n(j)} . 
\]
Note that 
\[
M_3 - \sum_{j=1}^m \partial \hat U_j =  \sum_{j=1}^m \sum_{i=1}^{n(j)} \hat \Gamma_j^i. 
\]
Note that for each $\Gamma_j^i$, there is an index $u(j,i) \in \{1,\dots,k\}$ so that $\Gamma_j^i$ intersects the sphere $\hat\Sigma_{u(j,i)}$ but not any of the other spheres. As such, we group the $\Gamma_j^i$ by $u(j,i) = 1,2,\dots,k$. We find that for $a \in \{1,2,\dots,k\}$,
\[
\sum_{\{i,j : u(i,j) = a\}} \hat \Gamma_j^i
\]
is a cycle of diameter at most $2R + \pi$. Indeed, 
\begin{equation}\label{eq:sum-hat-gamma-aij-bdry}
\partial \left[ \sum_{a=1}^k \sum_{\{i,j : u(i,j) = a\}} \hat \Gamma_j^i \right]= 0
\end{equation}
and if 
\[
\partial \left[ \sum_{\{i,j : u(i,j) = a\}} \hat \Gamma_j^i \right] \neq 0
\]
for some $a \in \{1,\dots,k\}$, then this boundary component intersects there sphere $\hat \Sigma_a$ (but none of the other spheres). As such, the other terms in the sum \eqref{eq:sum-hat-gamma-aij-bdry} could not cancel this term. This would be a contradiction.

\begin{figure}[htbp]
	\centering
	\includegraphics[width=\textwidth]{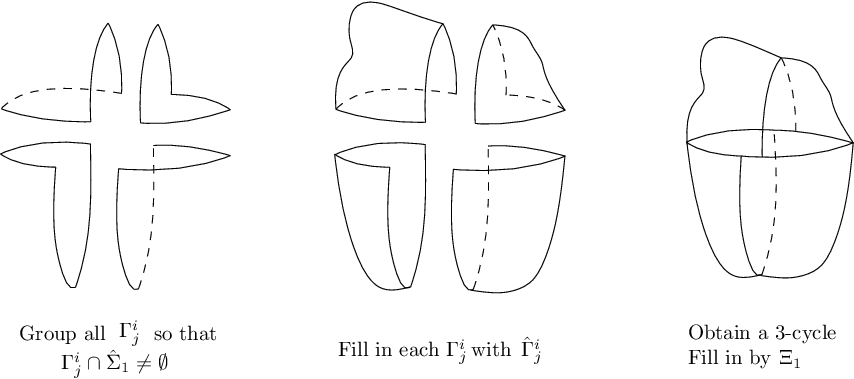}
	\caption{Filling in $\Gamma_j^i$ that intersecting $\hat \Sigma_1$}
	\label{pic5_fill_in}
\end{figure}

As such, there is $\tilde R>0$ (independent of $L$) so that by Proposition \ref{prop:fill-bd-diam} yet again, there is a $4$-chain $\Xi_a$ with 
\[
\partial \Xi_a = \sum_{i,j : u(i,j) = a} \hat \Gamma_j^i
\]
and extrinsic diameter at most $\tilde R$. Thus, we have written 
\[
M_3 = \partial \left[ \sum_{j=1}^m \hat U_j + \sum_{a=1}^k \Xi_a \right]
\]
where each term in the sum has diameter uniformly bounded as $L\to\infty$. This completes the proof.

\begin{comment}
\subsection{The Uryson width of compact conformally PSC $3$-manifold}
As a simple consequence of the slice--dice argument in Section \ref{section.slice.dice}, we obtain the following structure theorem for compact $3$-manifold with uniformly positive scalar curvature.

\begin{theo}\label{theo.3d.polyhedral.thm}
	Let $(M,g)$ be a closed Riemannian manifold. Suppose there exists $u>0$ on $M$ such that $R(u^4 g)>3$. Then $M$ admits a continuous map $\varphi$ to a graph $G$ such that $\diam_g(\varphi^{-1}(p))<10 \pi$ for any $p\in G$. 
\end{theo}

\begin{rema}
	By definition, this implies that first Uryson width (see, e.g. \cite[p. 41]{gromov2017questions})of $M$ is bounded by $10\pi$. Previously, this has been proved when $\pi_1(M)$ is trivial in \cite[Corollary 10.11]{GL:complete} and when $\pi_1(M)$ is finite or $\pi_1(M)=\ZZ$ in \cite{Katz1988diameter}. See also \cite[Conjecture 1A]{Guth2011volume}.
\end{rema}

\begin{proof}
	We first treat the case where $H^1(M)=0$. In this case, as in Section \ref{section.slice.dice}, we can find an exhaustion of $M$ by regions
	\[\Omega_1\subset\cdots\Omega_k\]
	such that properties (1)--(4) holds. We then construct a graph as follows. Let $v_1$ be a vertex corresponding to $\Omega_1$. Suppose we have constructed a 
\end{proof}

\end{comment}

\section{Proof of Theorem \ref{thm.pmt-bend}} \label{sec:proof-pmt-bend}

Let $n\ge 3$. For $X$ a $n$-manifold (compact or non-compact), suppose that $g$ is a complete metric on $M = T^n\# X$ with non-negative scalar curvature. By a result of Kazdan \cite{Kazdan:deform-PSC}, either $g$ is Ricci flat or $M$ admits a complete metric of positive scalar curvature. However, a complete Ricci flat metric on $T^n\# X$ is easily seen to be flat (this follows from the splitting theorem; for example, see \cite[Theorem 4]{CG:split}). As such, it suffices to consider the case of positive scalar curvature.  We will pass to an appropriate covering space and apply the $\mu$-bubble technique on the cover, after carefully choosing a weight function $h$. 

Fix $\ep>0$ small and define 
\[
\Xi : =  \{\vec x= (x_1,\cdots,x_n) \in \RR^n : |\vec x - \vec k| > \ep, \vec k \in \ZZ^n\}/\sim
\]
where $(x_1,\cdots,x_n) \sim (x_1+k_1,\cdots,x_n+k_n)$ for $k_1,\cdots,k_n \in \ZZ$. By assumption, there is a map $\Psi : \Xi \to  M$ so that $\Psi$ is a diffeomorphism onto its image. By scaling, we can assume that $R_g > 1$ on $\Psi(\Xi)$.

Observe that $M$ is (topologically) covered by $\hat M = (T^{n-1}\times \RR) \#_\ZZ X$ (unwind one of the $S^1$ factors in $T^n$). Define
\begin{equation}\label{eq.chi.hat}
\hat \Xi = \{\vec x=(x_1,\cdots,x_n) \in \RR^n : |\vec x - \vec k| > \ep, \vec k \in \ZZ^n\}/\sim
\end{equation}
where $(x_1,\cdots,x_n) \sim (x_1+k_1,\cdots, x_{n-1}+k_{n-1},x_n)$ and note that the map $\Psi$ lifts to $\hat \Psi :\hat \Xi \to \hat X$, a diffeomorphism onto its image $\hat M_0$. It is useful to write 
\[
\hat M = \hat M_0 \cup \left(\cup_{k\in\ZZ} \mathring X_k\right)
\]
where each $\mathring X_k$ is (topologically $X\setminus  B$ for an $n$-balls $B$ in $X$) attached to $\hat M_0 : = \hat\Psi(\hat \Xi)$ along small spheres centered at $(0,0,k)$. 

We now define a function $\rho_0 : \hat M \to \RR$ as follows. Define $\hat \Xi_{2\ep}$ as in \eqref{eq.chi.hat} but with $2\ep$ in the place of $\ep$. On $\Psi(\hat\Xi_{2\ep})$, we take $\rho_0(x_1,\cdots,x_n) = x_n$. On the annuli $\hat\Psi(\hat \Xi_{2\ep}\setminus \hat \Xi)$ centered at $(0,\cdots,0,k)$, interpolate between $x_n$ and $k+\tfrac 12$ (we can do this with uniformly $C^{1}$-norm independent of $k$). Then, on $\mathring X_k$ define 
\[
\rho_0(p) =\begin{cases} k + \tfrac 12 + \dist_g(p,\partial\mathring X_k)\quad k\ge 0;\\ k+\tfrac12 - \dist_g(p,\partial \mathring X_k)\quad k<0.\end{cases}
\]
We now define $\rho_1$ to be a smoothing of $\rho_0$. We can assume that $\rho_1\equiv k+\tfrac 12$ in a small neighborhood of $\partial \mathring X_k$. Since $\hat{M}$ is a covering space of $M$, there is $L>0$ so that 
\[|\Lip(\rho_1)|_g < L.\]
We may take $L$ larger if necessary to assume that $\tfrac{\pi L}{2} = J + \tfrac 34$ for some $J \in \NN$. 

We now define a function $h\in C(\hat M,[-\infty,\infty])$ as follows. On $\hat M_0 \cap \{ -\tfrac{\pi L}{2} \leq \rho_1 \leq \tfrac{\pi L}{2} \}$, we define 
\[
h(p) = - \tan(\tfrac{1}{L} \rho_1(p)). 
\]
On the rest of $\hat M_0$ we set $h = \pm \infty$ such that it is continuous to $[-\infty,\infty]$. We now define $h$ on $\mathring X_k$. When $k>J$, set $h=-\infty$ on $\mathring X_k$, and when $K<-J$, set $h=\infty$ on $\mathring X_k$. Now assume $|k| \leq J$.

For $0\le k\le J$ and
\[
p \in \mathring X_k \cap \left\{\rho_1 < k + \tfrac 12 + \frac{2L}{\tan(L^{-1}(k+\tfrac 12))}\right\},
\]
or $-J\le k<0$ and
\[
p \in \mathring X_k \cap \left\{\rho_1 > k + \tfrac 12 + \frac{2L}{\tan(L^{-1}(k+\tfrac 12))}\right\},
\]
 we set
\[
h(p) =  \frac{2L}{\rho_1(p) - (k+\tfrac 12 ) -  \frac{2L}{\tan ( L^{-1}(k+\tfrac 12))}} .
\]
Otherwise we set $h(p)=\pm\infty$ such that $h$ is continuous.

\begin{figure}[htbp]
	\centering
	\includegraphics[width=\textwidth]{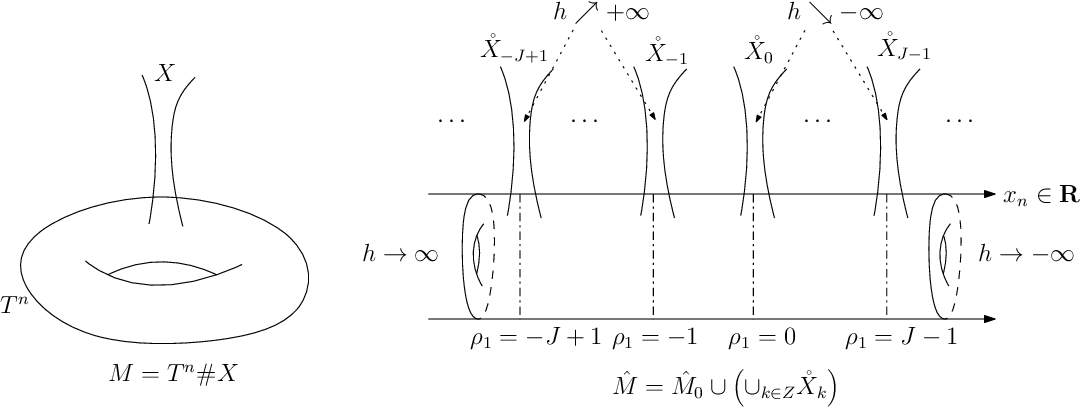}
	\caption{Unwrap $M=T^n\#X$ and construct a function $h$}
	\label{pic4_torus}
\end{figure}
 
We make several observations. First of all, since $-J \le k\leq J$, we see that 
\[
-\tfrac \pi 2 < L^{-1}(k+\tfrac 12) < \tfrac \pi 2. 
\]
Moreover, for $p \in \partial\mathring X_k$, we have that
\[
h(p) = - \tan (L^{-1}(k+\tfrac 12)) = - \tan(L^{-1}\rho_1(p)),
\]
and thus $h$ is Lipchitz across $\partial\mathring X_k$. Finally, if $0\le k \le J$, $p\in \mathring{X}_k$ and 
\[
\rho_1(p) \nearrow k + \tfrac 12 + \frac{2L}{\tan(L^{-1}(k+\tfrac 12))},
\]
we have that $h(p) \to -\infty$. Similarly, if $-J\le k \le 0$, $p\in \mathring{X}_k$ and
\[
\rho_1(p) \searrow k + \tfrac 12 + \frac{2L}{\tan(L^{-1}(k+\tfrac 12))},
\]
$h_1(p)\to \infty $. Thus $h$ is continuous.

Note that $\{|h| < \infty\}$ is compact. This is because this region is compact in $M_0$, only finitely many ends $\mathring X_k$ are included in this set, and in each $\mathring X_k$, the region where $\{|h|<\infty\}$ is bounded. 

\begin{lemm} We can smooth $h$ slightly to find a function $h\in C^\infty(\hat M)$ satisfying
\begin{equation}\label{eq:mu-bubble-ineq-h-psc}
R_{\hat g} + h^2 - 2|\nabla h| > 0
\end{equation}
on $\{|h|<\infty\}$. 
\end{lemm}
\begin{proof}

The function $h$ constructed above is smooth away from $\partial X_k$ (and Lipschitz there). As such, if we prove \eqref{eq:mu-bubble-ineq-h-psc} for function $h$ considered above, then we can easily find a smooth function satisfying \eqref{eq:mu-bubble-ineq-h-psc}.

We first check \eqref{eq:mu-bubble-ineq-h-psc} on $\hat M_0$. There, $R_{\hat g} > 1$. As such, we have that
\[
R_{\hat g} + h^2 - 2|\nabla h| > 1 + \tan^2(L^{-1}\rho_1(p)) - 2 \cos^{-2}(L^{-1}\rho_1(p)) > 0. 
\]
On the other hand, on $\mathring X_k$ (we assume that $k\geq 0$ as the $k<0$ case is similar), we only know that $R_{\hat g} > 0$. Nevertheless, we compute
\[
R_{\hat g} + h^2 - 2|\nabla h| >  \frac{4L^2 - 4L^2}{\left( \rho_1(p) - (k+\tfrac 12 ) -  \frac{2L}{\tan ( L^{-1}(k+\tfrac 12))}\right)^2} = 0. 
\]
This completes the proof. 
\end{proof}
We can thus consider $\mu$-bubbles with respect to the function $h$ we have just defined. We fix 
\[
\Omega_0 : =\left(\hat\Psi(\hat \Xi \cap\{x_n < -\tfrac 12\}) \cup (\cup_{k<0}\mathring X_k)\right)\cap \{|h|<\infty\}.
\]
We can minimize $\cA$ among all Cacioppoli sets $\Omega$ such that $\Omega\Delta \Omega_0$ is compactly contained in $\{|h|<\infty\}$ by the argument given in Proposition \ref{prop.existence.regularity} (with $u=1$). Denote by $\Omega$ the connected component of the minimizer containing $\{\rho_1=-J\}$. Since $n\le 7$, each component of $\partial\Omega$ is compact and regular. By the stability inequality for $\cA$ from Lemma \ref{lemm:2nd-var} (with $u=1$) and \eqref{eq:mu-bubble-ineq-h-psc}, we see that $\Sigma = \partial \Omega$ satisfies 
\begin{equation}\label{eq:def-to-psc-12}
\int_\Sigma \left( |\nabla \varphi|^2 + \tfrac 12  R_\Sigma \varphi^2 \right) d\cH^{n-1} > 0
\end{equation}
for all $\varphi \in C^\infty(\Sigma)$. 

We can find a compact region $\hat M' \subset \hat M$ with smooth boundary so that $\partial\Omega \subset \hat M'$. Furthermore, we can arrange that $\partial \hat M' \cap \hat M_0 = \hat\Psi(\{z = \pm (J+1)\})$. Note that the other boundary components of $\hat M'$ thus lie completely in some $\mathring X_k$. 

In particular, 
$\partial\hat M \setminus \hat M_0$ bounds some compact manifold with boundary. Cap these components off and then glue the $\{z=J+1\}$ and $\{z=-(J+1)\}$ tori to each other. We thus find a manifold $\tilde M$ diffeomorphic to $T^n \# \tilde X$ for $\tilde X$ closed and $\Sigma^{n-1}\subset \tilde M$ a hypersurface that is homologous to $[T^{n-1}\times \{*\}]\in H_{n-1}(\tilde M)$ that satisfies \eqref{eq:def-to-psc-12}. (We have not constructed a metric on $\tilde M$, but this does not matter in the remaining part of the argument, all we need is the topology of $\tilde M$ and $\Sigma$ as well as the fact that $\Sigma$ satisfies \eqref{eq:def-to-psc-12}.) 

We claim that this leads to a contradiction following the argument in \cite{SY:descent,SY:sing-PMT}. Indeed, on the one hand \eqref{eq:def-to-psc-12} implies that (each component of) $\Sigma$ has positive first eigenvalue of the conformal Laplacian, and thus admits a metric of positive scalar curvature. On the other hand, we can pull back the $1$-forms $dx^1,\dots,dx^{n-1}$ along the map $\pi: (T^n\#\tilde X) \to T^n$ to find $1$-forms $\omega^1,\dots,\omega^{n-1}$ so that 
\[
\int_\Sigma \omega^1 \wedge\dots\wedge \omega^{n-1} \neq 0. 
\]
(this follows from the fact that $\Sigma$ is homologous to $T^{n-1}\times \{*\}$). The proof can now be completed using the inductive method of \cite{SY:descent,SY:sing-PMT}.

\section{Schoen-Yau-Schick manifolds}\label{section.SYS}
In this section, we briefly indicate an extension of Theorem \ref{thm.pmt-bend} to manifolds in the form of $(M\times S^1)\#X$, where $M$ is a Schoen--Yau--Schick manifold (abbreviated as SYS manifold, following the definition of Gromov \cite[Section 5]{Gromov:metric-inequalities}). These manifolds was first considered in the celebrated work \cite{SY:descent} of Schoen-Yau where the inductive descent argument was introduced. In \cite{SY:descent}, these manifolds are said to be of class $C_n$. We recall the definition here.

\begin{defi}[\cite{SY:descent}\cite{Schick1998counterexample}\cite{Gromov:metric-inequalities}\cite{SY:sing-PMT}]\label{defi.SYS}
	Let $n\ge 2$. A compact orientable $n$-manifold $M$ is called Schoen--Yau--Schick, if there exist $n-2$ integer homology classes $h_1,\cdots,h_{n-2}\in H^1(M)$ such that $\sigma= h_1\frown \cdots\frown h_{n-2}\frown [M]\in H_2(M,\ZZ)$ is non-spherical. That is, $\sigma$ is not contained in the image of the Hurewicz homomorphism $\pi_2(M)\to H_2(M)$.
\end{defi}

For example, the torus is an SYS manifold. Using minimal surface and induction descent argument, Schoen--Yau in \cite{SY:descent} proved that SYS manifolds of dimension at most $7$ does not admit positive scalar curvature metrics. In \cite{Schick1998counterexample}, Schick constructed an SYS manifold which is a counterexample to the unstable Gromov-Lawson-Rosenberg conjecture. 

Let $n\le 6$, and $M^n$ be an SYS manifold. By passing to the cover $M\times \RR$ and constructing the same functions $\rho_0,\rho_1$ and $h$ as in Section 6, we can extend Theorem \ref{thm.pmt-bend} to the following.

\begin{theo}\label{thm.SYS}
	Let $2\le n\le 6$, $M^n$ be an Schoen--Yau--Schick manifold. For any $(n+1)$-manifold $X$, the connected sum $(M\times S^1)\#X$ does not admit a complete metric of positive scalar curvature.
\end{theo}

\begin{rema}
After the first version of this paper first appeared, S.\ Chen \cite{Chen:SYS} generalized the proof of Theorem \ref{thm.SYS} to show that there is no complete positive scalar curvature metric on $M\# X$ for $M^n$ a Schoen--Yau--Schick manifold, $3 \leq n\leq 7$ and $X^n$ an arbitrary manifold. 
\end{rema}

\appendix
\section{The Lesourd--Unger--Yau reduction}

Schoen--Yau demonstrated that the Liouville theorem stated here in Corollary \ref{coro.SY-Liouville} is equivalent to a certain positive (non-negative) mass theorem for asymptotically flat manifolds with other ends \cite[\S 4]{SY:conf-flat} (cf.\ \cite[\S VI]{SY:lectures}) The usual positive mass theorem can be reduced to the non-existence of a positive scalar curvature metric on $T^n \# X$ (for $X$ compact) by an argument of Lohkamp. However, in general this argument involves global conformal deformations as first developed by Schoen--Yau in their proof of the positive mass theorem \cite{SY:PMT1}. 

As such, it is not immediate that in the context of Corollary \ref{coro.SY-Liouville}, it suffices to prove Theorem \ref{thm.pmt-bend}. This has been verified by Lesourd--Unger--Yau \cite{LUY:liouville} (cf.\ \cite[\S 4]{SY:conf-flat}). For completeness, we indicate a proof of their reduction here.

By work of Schoen--Yau \cite[Proposition 3.3]{SY:conf-flat}, to prove Corollary \ref{coro.SY-Liouville}, it suffices to consider $3\leq n\leq 6$ and a complete Riemannian $n$-manifold $(M^n,g)$ with $R_g\geq 0$ and a conformal map $\Phi : M\to S^n$. Denote by $g_0$ the round metric on $S^n$ and write $\Phi^* g_0 = |\Phi'|^2 g$. 

By \cite[Corollary 1.3]{SY:conf-flat}, the conformal Laplacian
\[
L_g = -\Delta_g + \frac{n-2}{4(n-1)} R_g
\]
admits a minimal Greens function $G>0$ satisfying $L_gG = \delta_{x_0}$ for every $x_0 \in M$. Denote by $G_0$ the Green's function for $L_0$ (the conformal Laplacian on $(S^n,g_0)$) based at $\Phi(x_0) = y_0$. It is a standard fact that 
\[
G_0^{\frac{4}{n-2}} g_0 = \pi^*(g_{\RR^n}),
\]
where $\pi : S^n \setminus\{y_0\} \to \RR^n$ is the stereographic projection map. Pull this equation back by $\Phi$ to and define 
\[
\overline g : = (G_0 \circ \Phi)^{\frac{4}{n-2}} |\Phi'|^2 g = (\Phi \circ \pi)^* g_{\RR^n}
\]
on $M\setminus \Phi^{-1}(y_0)$. Thus, $\overline g$ is flat (but likely incomplete). Observe, however that $x_0$ has a neighborhood $U$ so that $\overline g|_{U\setminus \{x_0\}}$ is isometric to $g_{\RR^n} |_{B_R(0)^c}$ via the map $\Phi \circ \pi$. 

Motivated by this fact, we define
\[
\overline G = |\Phi'|^{\frac{n-2}{2}} G_0 \circ \Phi
\]
and check that (see \cite[(1.1)]{SY:conf-flat})
\[
L_g\overline G = \sum_{P \in \Phi^{-1}(y_0)} |\Phi'(P)|^{\frac{n+2}{2}} \delta_P. 
\]
By construction (cf.\ \cite[Corollary 1.3]{SY:conf-flat}) we have that $|\Phi'(x_0)|^{\frac{n+2}{2}} G \leq \overline G$. 

We now define $\overline g_\ep = (G+\ep)^\frac{4}{n-2} g$. Note that for $\ep>0$, this is complete and (cf.\ \cite[(1.2)]{SY:conf-flat})
\[
R_{\overline g_\ep} = \frac{4(n-1)}{n-2} (G+\ep)^{-\frac{n+2}{n-2}}L_\ep(G+\ep) =  \ep  (G+\ep)^{-\frac{n+2}{n-2}} R_g \geq 0. 
\]
Moreover, $\overline g_\ep = ((G+\ep)/G)^\frac{4}{n-2} \overline g = (\Phi \circ \pi)^*(v_\ep^{\frac{4}{n-2}} g_{\RR^n})$, where 
\[
v_\ep \circ \Phi \circ \pi = G/\overline G  + \ep/\overline G : = (v + \ep \tilde v) \circ \Phi \circ \pi. 
\]
Observe that $v$ is harmonic on $B_R(0)^c\subset \RR^n$ and limits to $|\Phi'(x_0)|^{\frac{n+2}{2}}$ as $x\to\infty$. This follows from the fact that 
\[
G, |\Phi'(x_0)|^{-\frac{n+2}{2}}G_0 = c_nr^{n-2}(1+o(1))
\]
near $x_0$, for some dimensional constant $c_n > 0$. This is easily proven by a scaling argument (cf.\ \cite[\S V, Theorem 3.5]{SY:lectures}, but note that the full power of that result is not needed here).

Moreover, $\tilde v(x) = c_n^{-1} |\Phi'(x_0)|^{\frac{n+2}{2}} |x|^{2-n} + O(|x|^{1-n})$ as $x\to\infty$. Thus, we find that $v_\ep$ can be expanded as
\[
v_\ep(x) = A + B_\ep |x|^{2-n} + O(|x|^{1-n})
\]
for some $A >0$. By a well-known argument due to Lohkamp (see \cite[Proposition 5.4]{SY:sing-PMT}), if $B_\ep < 0$, we can modify $\overline g_\ep$ only in the chart $(\Phi\circ\pi)(B_R(0)^c)$ and obtain a complete metric $\overline g_\ep'$ on $M\setminus \Phi^{-1}(y_0)$ with $\overline g_\ep'$ flat on $(\Phi\circ\pi)(B_{2R}(0)^c)$ and so that $\overline g_\ep'$ has non-negative scalar curvature. Quotienting the sides of a large box in the flat region thus yields a contradiction to Theorem \ref{thm.pmt-bend}. 

Thus, we find that $B_\ep \geq 0$. Sending $\ep\to 0$, we thus find a harmonic function $v$ on $B_R(0)^c \subset \RR^n$ with
\[
v(x) = A + B |x|^{2-n} + O(|x|^{1-n})
\]
for some $A>0,B\geq 0$. Observe that $v(x) \leq A$ on $B_R(0)^c$. This follows from the construction of $v$ and the fact that $G$ was a minimal Greens function as discussed above. We can now apply the maximum principle to conclude that
\[
- B|x|^{2-n} + O(|x|^{1-n}) = A-v(x) \geq \left(\min_{\partial B_{2R}(0)} (A-v)\right)|x/(2R)|^{2-n}
\]
Because $B \geq 0$, we thus see that $v(x) \equiv A$. This implies that $G=\overline G$, completing the proof. 

\bibliography{bib}
\bibliographystyle{amsplain}

\end{document}